\documentclass[11pt,reqno]{amsart} 
\usepackage[a4paper,left=24mm,right=24mm,top=32mm,bottom=30mm]{geometry} 
\usepackage{epic,eepic,verbatim,amsmath,amssymb} 
\usepackage{graphicx} 
\usepackage[pdftex]{hyperref} 
\newcommand{\pd}{\partial}                 
\renewcommand{\d}{{\rm{d}}}                
\newcommand{\dg}{\, \d \sigma(x)}    
\newcommand{\dgxi}{\, \d \sigma(\xi)} 
\newcommand{\dx}{\, {\rm{d}} x}            
\newcommand{\D}{\nabla}                    
\newcommand{\R}{\mathbb{R}}                
\newcommand{\Rn}{\mathbb{R}^n}                
\newcommand{\N}{\mathbb{N}}                
\newcommand{\Aset}{\mathcal A}                
\newcommand{\SL}{\Delta_{\Gamma}}          
\newcommand{\SLT}{\Delta_{\tilde \Gamma}}  
\newcommand{\SG}{\nabla_{\Gamma}}          
 
\newcommand{\scpgh}[2]{\langle #1,#2 \rangle_{\Gamma_h}}
 
\newcommand{\matr}[1]{{\boldsymbol #1}}  
\newcommand{\taum}{\Delta t_m}   

\renewenvironment{proof}[0] {\noindent{\em Proof.}}{\hfill \qed\\[1ex] }

\theoremstyle{plain}
\numberwithin{equation}{section}
\newtheorem{lemma}{Lemma}[section]
\newtheorem{theorem}[lemma]{Theorem}
\newtheorem{proposition}[lemma]{Proposition}

\theoremstyle{definition}
\newtheorem{remark}[lemma]{Remark}

\begin{document} 

\title[Turing instabilities]
{Turing instabilities in a mathematical model for signaling networks}

\author[A. R\"atz]{Andreas R\"atz}
\email{andreas.raetz@tu-dortmund.de} 

\author[M. R\"oger]{Matthias R\"oger}
\email{matthias.roeger@tu-dortmund.de} 



\date{\today}
\maketitle

\begin{abstract}
GTPase molecules are important regulators in cells that continuously run through an activation/deactivation and membrane-attachment/membrane-detachment cycle. Activated GTPase is able to localize in parts of the membranes and to induce cell polarity. As feedback loops contribute to the GTPase cycle and as the coupling between membrane-bound and cytoplasmic processes introduces different diffusion coefficients a Turing mechanism is a natural candidate for this symmetry breaking. We formulate a mathematical model that couples a reaction--diffusion system in the inner volume to a reaction--diffusion system on the membrane via a flux condition and an attachment/detachment law at the membrane. We present a reduction to a simpler non-local reaction--diffusion model and perform a stability analysis and numerical simulations for this reduction. Our model in principle does support Turing instabilities but only if the lateral diffusion of inactivated GTPase is much faster than the diffusion of activated GTPase.
\end{abstract}

\section{Introduction}
\label{sec:intro}
GTP-binding proteins (GTPases) are crucially involved in many processes in cells such as membrane traffic, cellular transport, signal transduction, or cytoskeleton organization \cite{TaSM01,JaHa05}. Common to the diverse families of GTPase is the cycling between an active and an inactive state. Besides the activation-inactivation cycle there is also a spatial cycle: in the cytosol almost all GTPase is inactive whereas the active state is only present at the membrane. Reaction and diffusion processes both in the cytosolic volume and on the membrane surfaces as well as membrane attachment and detachment contribute to the proper function of GTPase molecules. 

For different GTPase localization into subcellular compartments has been observed and has been recognized as crucial for its function. Cluster formation of activated small GTPase Cdc42 precedes the budding of yeast \cite{PaBi07}, other  small GTPase of the Rho-subfamily are known to form micro-domains on continuous membranes \cite{PfAi04,SRNR00}. Such a transition from a homogeneous distribution to a polarized state is often key for the formation and maintenance of complex structures. The emergence of localized structures is typically driven by a continuous input of energy \cite{NiPr77}. Turing \cite{Turi52,Murr90} pioneered models for symmetry breaking by diffusion-driven instabilities. These are based on a slowly diffusing self-activator and a highly diffusive antagonist \cite{KoMe94}. Self-activation is typically present by some kind of feedback. In activator--substrate-depletion type Turing mechanisms the production of the activator induces a decrease of the substrate. Diffusion-driven instabilities typically require large differences in the diffusion coefficients of the activator and its antagonist. In many biological applications this is not realistic and Turing type mechanisms can therefore not explain symmetry breaking events. In our context, however,  cytosolic diffusion is typically much faster than lateral diffusion. Coupled systems of 2D and 3D reaction--diffusion processes might therefore be a candidate for a realistic Turing mechanism.

Distinct mathematical models for the GTPase cycle have been proposed
and analyzed, with diverse conclusions. A general model for signaling
molecules in a cell that cycle between a non-recruiting cytosolic
state and a recruiting membrane-bound one has been evaluated in
\cite{AAWW08}. There the emergence of cell polarity has been
demonstrated for an intrinsically \emph{stochastic} mechanism for
self-activation by positive feedback. A corresponding deterministic
model in contrast has been shown not to produce any heterogeneous
pattern. However, the deterministic PDE model in \cite{AAWW08} does
not directly treat processes in the cytosol as all variables are
membrane bound and all reactions are local.  A complex PDE model that
accounts for chemical reactions, membrane-cytoplasm exchange and
diffusion is given in \cite{GoPo08}. There a scaling factor accounts
for the different volume of a thin 3D membrane layer and the inner
volume. Variables representing membrane bound molecules and variables
representing cytosolic quantities however both have the same domain of
definition. Numerical simulations and a linear stability analysis show
that the model allows for a Turing mechanism. The formation of
micro-domains in a GTPase cycle model are also studied in
\cite{BDKR07}. Here the equations are formulated on a flat membrane
surface. It is shown that no Turing pattern can occur unless an extra
flux term is included. This flux term accounts for interactions
between GTPase and membrane proteins and represents a phase separation
type energy gradient. As an alternative explanation for the emergence
of cell polarity in GTPase mediated processes a `wave-pinning'
mechanism is proposed in \cite{MoJE08}. A two component system for the
nucleotide cycle is suggested with a Hills type non-linearity that
leads to a bistability. Domains are formed by emerging traveling waves
that are stopped by a decreased supply of non-activated GTPase.

Our goal is to introduce a model for the GTPase cycle with an improved coupling of processes with different dimensionalities. We investigate whether a Turing type instability -- of activator--substrate depletion type -- could possibly explain the localization of activated GTPase on the membrane. In Section \ref{sec:model} we will first formulate our mathematical model and derive  a reduction that only incorporates membrane-bound active and membrane-bound inactive GTPase. We perform a stability analysis and numerical simulations for this reduction. The explicit dimensional coupling in the full model is still reflected by the appearance of a \emph{non-local} term. We show in Section \ref{sec:turing} that for this model Turing patterns are possible. Our numerical simulations in Section \ref{sec:numerics} confirm this result and shed some light on the kind of patterns that are supported by our model and the influence of changes in different parameters. We develop here a general numerical scheme that can be extended to general membrane geometries and to more involved coupling laws. In Section \ref{sec:no-turing} we investigate -- even for a more general class of similar models -- whether for equal diffusion constants of activator and substrate, \emph{i.e.} activated and non-activated GTPase, Turing pattern are possible. Our results will finally be discussed in Section \ref{sec:discussion}. 
\section*{Acknowledgment}
We would like to thank Roger Goody and Yaowen Wu from the Max-Planck institute for molecular physiology for introducing us to the biochemistry of signaling networks.
\section{Model Description}\label{sec:model}
\subsection{Mechanistic description of the GTPase cycle}
Here we briefly review the key steps of the GTPase cycle as indicated
in Fig. \ref{fig:cycle}. Chemically, the difference between the active and inactive state of the GTPase is that in the active state \emph{guanine-tri-phosphate} (GTP) is bound whereas \emph{guanine-di-phosphate} (GDP) is bound in the inactive state. Only activated GTPase interacts with downstream effectors. Activation of a GTPase is by exchange of GDP by GTP, inactivation by hydrolysis and dephosphorylation of GTP to GDP. Both processes are intrinsically very slow and need the catalyzation by a GEF (\emph{guanine exchange factor}) and GAP protein (\emph{GTPase activating protein}), respectively \cite{GuAG05,BoRW07}. Cytosolic GTPase can only be found in complex with a displacement inhibitor (GDI) that prevents the binding of GTPase to the membrane. As the affinity of GTPase towards GDI is much higher when GDP is bound, predominantly the inactive state occurs in the cytosol \cite{GoRa05}. How GDP-bound GTPase is released from the complex with GDI and how it associates to the membrane is less clear, mediation by a GDI displacement factor (GDF) has been proposed as a possible mechanism \cite{Pfef03}. For several GTPase positive feedback loops have been identified that support the activation of GTPase. Activated GTP-Rab5 is known to recruit a cytosolic GEF-effector complex (Rabex5 and Rabaptin5) to the membrane and to increase the activity of the GEF \cite{GrON06}. A similar feedback loop has been found for activated Cdc42 GTPase \cite{WAWL03}. In the following we formulate a mathematical model that reflects the key features of a GTPase cycle. Our main focus is on the treatment of the dimensional coupling and less on a detailed description of the reaction kinetics.
\begin{figure}[h]
  \centering
  \includegraphics[width=0.95\textwidth]{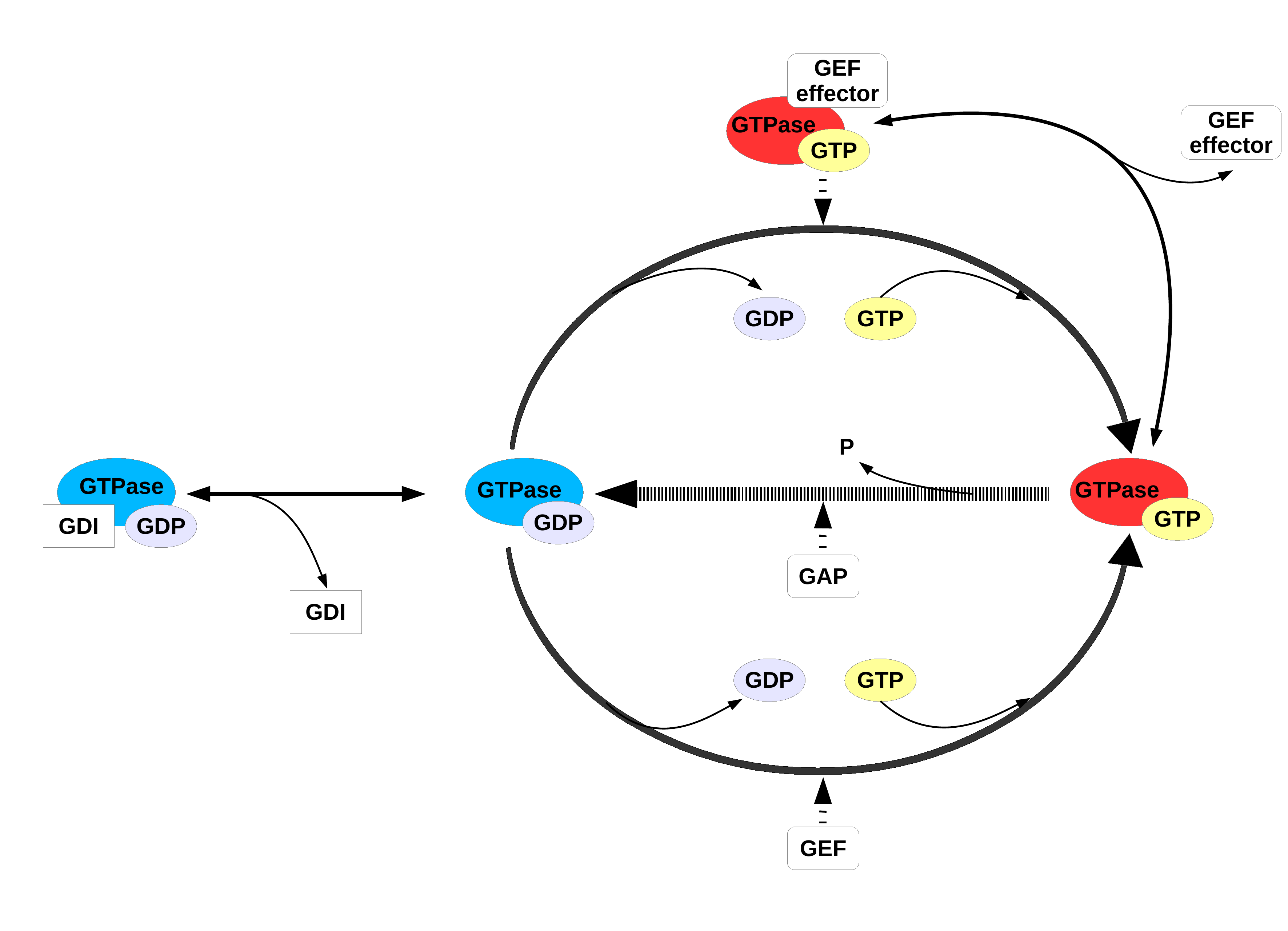}
  \caption{\label{fig:cycle} The GTPase reaction cycle:
    The activation of GDP-bound GTPase is either catalyzed by GEF (lower semi
    circle) or by an effector--GEF--GTP-GTPase complex (upper semi circle).
    The inactivation of GTP-bound GTPase is catalyzed by GAP.  Further reactions that are depicted are
    GDP-GDPase--GDI complex formation/dissociation and
    effector--GEF--GTP-GTPase complex formation/dissociation. See the text for additional information.}
\end{figure}
\subsection{The mathematical model}
We restrict ourselves to the investigations of processes in the cytosol and at the outer plasma membrane only. Inner organelles with additional membrane boundaries could be included as well. The cytosolic volume of a cell is represented by a bounded, connected, open domain $B \subset \R^3$ and the cell membrane by the boundary of $B$ that we assume to be given by a smooth, closed two-dimensional surface $\Gamma := \pd B$ without boundary. In addition we
fix a time interval of observation $I:=[0,T] \subset \R$. We formulate a system
of PDE's for the following unknowns:
\begin{align*}
	V &: \overline B \times I \to\R &&\text{concentration of cytosolic
  GDP-GTPase (in complex with GDI)},\\
	v &:\Gamma \times I \to \R &&\text{concentration of membrane-bound GDP-GTPase},\\
	u &:\Gamma \times I \to \R &&\text{concentration of membrane-bound GTP-GTPase},\\
	m &:\Gamma \times I \to \R &&\text{concentration of membrane-bound,
  effector--GEF--GTP-GTPase complex},\\
	g&:\Gamma \times I \to \R &&\text{concentration of membrane-bound GEF.}
\end{align*}
We prescribe initial conditions at time $t=0$,
\begin{align*}
	&V(\cdot, 0) \,=\, V_0,\quad v(\cdot, 0) \,=\, v_0,\quad
	u(\cdot, 0) \,=\, u_0,\quad m(\cdot, 0) \,=\, m_0, \quad
	g(\cdot, 0) \,=\, g_0,\\
	&V_0\,:\, B\,\to\, \R,\quad v_0,u_0,m_0,g_0\,:\,\Gamma\,\to\,\R.
\end{align*}
The coupling condition between cytosolic and membrane processes involves a Neumann boundary condition for $V$ that is specified below. Physical units are given by
\begin{equation*}
  [V] = \frac{\text{mol}}{\text{m}^3}, \quad 
  [u] = [v] = [m] = [g] = \frac{\text{mol}}{\text{m}^2}.
\end{equation*}
Much more extended sets of variables could be considered here. In particular we do not explicitly take into account the effector and GAP concentrations. Catalyzation of the inactivation process will be described implicitly.  
\subsubsection{Reaction Kinetics}
We assume simple mass action kinetics or a Michaelis--Menten type law for catalyzed reactions. For the change of concentration of the above variables due to reactions we prescribe the following equations. Our choices here  are similar to the more general model in \cite{GoPo08}.

The concentration of membrane-bound GDP-GTPase is decreased by the activation process, which is catalyzed by both GEF and the effector--GEF--GTP-GTPase complex. For the corresponding rates we assume that they are proportional to the GDP-GTPase concentration and the concentrations of the catalysts. Vice versa GDP-GTPase is produced by the inactivation of GTP-GTPase. Since we have not taken the GAP concentration into account, we here assume a Michaelis--Menten law for the kinetics. The change of $v$ due to activation and inactivation we therefore describe by
\begin{align}
  \label{eq:v_kin}
  [\pd_t v]_{\text{reaction}} &= - k_1 v g - k_2 v m + k_3 \frac{u}{u
    + k_4}  \quad \text{on} \quad \Gamma
  \times I.
\end{align}
For the change in $u$ we have in addition to the processes above the production of $u$ by dissociation of the effector -- GEF -- GTP-GTPase complex and the loss due to the formation of this complex. We model this by the equation
\begin{align}
  \label{eq:u_kin}
  [\pd_t u]_{\text{reaction}} &= k_1 v g + k_2 v m - k_3 \frac{u}{u + k_4} -
  k_5 ug + k_{-5}m
  \quad \text{on} \quad \Gamma \times I.
\end{align}
Correspondingly complex formation and dissociation lead to the following laws for the concentration of the complex and GEF (or rather a GEF--effector complex as we do not take explicitly into account the effector).
\begin{align}
  \label{eq:m_kin}
  [\pd_t m]_{\text{reaction}} &= k_5 u g - k_{-5} m \quad
  \text{on} \quad \Gamma \times I,\\
  \label{eq:g_kin}
  [\pd_t g]_{\text{reaction}} &= - k_5 u g + k_{-5}m \quad
  \text{on} \quad \Gamma \times I.
\end{align}
Units for the reaction rates $k_i$ are given by
\begin{equation*}
  [k_1] = [k_2] = [k_5] = \frac{\text{m}^2}{\text{mol}\cdot\text{s}}; \quad
  [k_3] = \frac{\text{mol}}{\text{m}^2\text{s}} ; \quad
  [k_4] = \frac{\text{mol}}{\text{m}^2} ; \quad
  [k_{-5}] = \frac{1}{\text{s}}.
\end{equation*}

\subsubsection{Simplified Kinetics}
For later use in the mathematical analysis we use a quasi--steady state approximation for the complex
formation to do a first reduction of our kinetic laws. We assume
\begin{equation}
  m = \frac{k_5}{k_{-5}}ug \label{eq:m-stst}
\end{equation}
and use GEF--conservation
\begin{equation}
  m + g = \text{const} =: \bar{g}_0, \label{eq:gef-cons}
\end{equation}
where $[\bar{g}_0] = [g] = \frac{\text{mol}}{\text{m}^2}$. If we take initial data $m_0=0$ we have $\bar{g}_0=g_0$. Equations \eqref{eq:m-stst}, \eqref{eq:gef-cons} yield
\begin{align*}
  m &= \frac{K_5ug_0}{1 + K_5u},\\
  g &= g_0\left(1 - \frac{K_5u}{1 + K_5 u}\right)
\end{align*}
with $K_5:=\frac{k_5}{k_{-5}}$ and $[K_5] =
\frac{\text{m}^2}{\text{mol}}$. From this, one obtains simplified
rate equations for the change due to reactions
\begin{align*}
  [\pd_t v]_{\text{reaction}} &= - k_1 v g_0\left(1 - \frac{K_5u}{1 +
      K_5 u}\right) - k_2 v \frac{K_5ug_0}{1 + K_5u} + k_3 \frac{u}{u
    + k_4}  \quad \text{on} \quad \Gamma \times I,\\ 
  [\pd_t u]_{\text{reaction}} &= k_1 v g_0\left(1 - \frac{K_5u}{1 +
      K_5 u}\right) + k_2 v \frac{K_5ug_0}{1 + K_5u} - k_3 \frac{u}{u +
    k_4} = - [\pd_t v]_{\text{reaction}}
  \quad \text{on} \quad \Gamma \times I.
\end{align*}

\subsubsection{Diffusion}
We describe cytosolic diffusion of the (inactive) GTPase in $B$ by the standard Laplace diffusion operator $\Delta$ and a diffusion constant $D>0$. Lateral diffusion on the membrane is described by the Laplace--Beltrami--operator  $\SL$ (which is the generalization of the ordinary Laplacian to surfaces \cite{DiHS10}) and diffusion constants $d_u$ and $d_v$ for the active and inactive membrane-bound GTPase concentrations, respectively.
\begin{align}
  \label{eq:V_diff}
  [\pd_t V]_{\text{diffusion}} &= D \Delta V \quad \text{in} \quad B
  \times I,\\
  \label{eq:u_diff}
  [\pd_t u]_{\text{diffusion}} &= d_u \SL u \quad \text{on} \quad
  \Gamma \times I,\\
  \label{eq:v_diff}
  [\pd_t v]_{\text{diffusion}} &= d_v \SL v \quad \text{on} \quad \Gamma
  \times I.
\end{align}
\subsubsection{Membrane attachment and detachment}
We describe the association and dissociation of inactive GTPase at the membrane by a flux boundary condition for $V$ at $\Gamma$
\begin{equation}
  \label{eq:flux1}
  -D \nabla V \cdot \nu = q  \quad \text{on} \quad \Gamma,
\end{equation}
where $\nu$ denotes the outer normal to $B$ at $\Gamma$. For the flux $q$ we formulate a constitutive equation: membrane attachment is treated as a reaction between cytosolic GTPase and a free site on the membrane and modeled by a Langmuir rate law \cite{Kell09}. Detachment is taken proportional to the inactive GTPase concentration, which together gives the equation
\begin{equation}
  \label{eq:flux2}
  q = b_6 \frac{|B|}{|\Gamma|}V (c_{\max} - u - v)_+ - b_{-6}v.
\end{equation}
Here $c_{\max}$ denotes a saturation value and $b_6, b_{-6}$ are
sorption coefficients. By $ (c_{\max} - u - v)_+$ we denote the positive part of $c_{\max} - u - v$ as adsorption stops when the saturation value is reached. $|B|$ and $|\Gamma|$ denote the
$3$-dimensional volume of $|B|$ and the $2$-dimensional
surface area of $\Gamma$, respectively. In \eqref{eq:flux2} $\frac{|B|}{|\Gamma|}V$ has to be understood as the trace of the cytosolic GDP-GTPase concentration and has units $\frac{\text{mol}}{\text{m}^2}$. The units of the various coefficients are given by 
\begin{equation*}
  [D] = [d_u] = [d_v] = \frac{\text{m}^2}{\text{s}}; \quad 
  [b_6] = \frac{\text{m}^2}{\text{mol} \cdot \text{s}}; \quad 
  [b_{-6}] = \frac{1}{\text{s}}.
\end{equation*}

\subsubsection{Reaction, Diffusion, and attachment/detachment}
Taking reaction and diffusion into account, one obtains the following
model, which is the basis of further considerations 
\begin{align}
  \label{eq:V-dim}
  \pd_t V &= D \Delta V \quad \text{in} \quad B
  \times I,\\
  \label{eq:u-dim}
  \pd_t u &= d_u \SL u + k_1 v g_0\left(1 - \frac{K_5u}{1 +
      K_5 u}\right) + k_2 v \frac{K_5ug_0}{1 + K_5u} - k_3 \frac{u}{u +
    k_4}\quad \text{on} \quad
  \Gamma \times I,\\
  \label{eq:v-dim}
  \pd_t v &= d_v \SL v - k_1 v g_0\left(1 - \frac{K_5u}{1 +
      K_5 u}\right) - k_2 v \frac{K_5ug_0}{1 + K_5u} + k_3 \frac{u}{u
    + k_4} + q\quad \text{on} \quad \Gamma
  \times I
\end{align}
with the flux conditions \eqref{eq:flux1}, \eqref{eq:flux2}. The model
satisfies conservation of GTPase in the form
\begin{equation*}
  \frac{\rm{d}}{\rm{d} t} \left(\int_B V \dx + \int_\Gamma (u + v)
    \dg\right) =0,
\end{equation*}
where $\dg$ denotes integration with respect to the surface area measure. This equation confirms that the total number of cytosolic inactive plus membrane-bound active and inactive GTPase is constant over time.
\subsubsection{Non--Dimensionalization}
For $x \in \overline B$ and $t \in I$, we introduce non-dimensional
coordinates $\xi$ 
\begin{align*}
  \xi := \frac1{R}x, 
\end{align*}
where $R > 0$ denotes a typical length, e.g. half the diameter of the
cell. We represent this length as $R=\sqrt{\gamma}\,\mathbb{I}$ with
$\gamma>0$ and $\mathbb{I}=1\text{m}$ denoting the unit
length. Furthermore, we use a dimensionless time
\begin{equation*}
  \tau := \frac{d_u}{R ^2} t.
\end{equation*}
This leads to transformed domain $\tilde B := \{\xi \in \R^3 : R
\xi \in B \}$, $\tilde \Gamma := \pd \tilde B$ and time interval $\tilde
I := [0,\frac{d_u}{R ^2}T]$. Non-dimensional Rab concentrations
are defined through 
\begin{equation*}
  \tilde V := \frac{R}{c_{\max}} V , \quad 
  \tilde u := \frac{1}{c_{\max}} u, \quad
  \tilde v := \frac{1}{c_{\max}} v. 
\end{equation*}
We introduce dimensionless quantities
\begin{gather*}
  a_1 := \frac{\mathbb{I} ^2}{d_u}k_1g_0, \quad 
  a_2 := \frac{1}{K_5c_{\max}}, \quad
  a_3 := \frac{k_2}{k_1}a_1, \quad
  a_4 := \frac{\mathbb{I} ^2}{d_u c_{\max}} k_3, \quad
  a_5 := \frac{k_4}{c_{\max}}, \\
  a_6 := \frac{\mathbb{I}^2 b_6}{d_u} c_{\max}\frac{|B|}{|\Gamma|R},\quad
  a_{-6} := \frac{\mathbb{I} ^2 b_{-6}}{d_u},\quad
  d := \frac{d_v}{d_u}, \quad
  \tilde D := \frac{D}{d_u}.
\end{gather*}
Note that $\frac{|B|}{|\Gamma|R}$ is scale invariant in the sense that multiplying the system size by a constant factor $\alpha>0$ does not affect this value. In particular all above constants are independent of the system size, which is solely represented by the dimensionless quantity $\gamma$. With these
definitions, one easily verifies
\begin{align}
  \label{eq:V_tilde}
  \pd_\tau \tilde V &= \tilde D \Delta_\xi\tilde V \quad \text{in} \quad
  \tilde B  \times \tilde I,\\
  \label{eq:u_tilde}
  \pd_\tau \tilde u &= \SLT \tilde u + \gamma\Big(\left(a_1 + (a_3 -
    a_1)\frac{\tilde u}{a_2 + \tilde u} \right)\tilde v - a_4
  \frac{\tilde u}{a_5 +\tilde u}\Big)\quad \text{on} \quad
  \tilde \Gamma \times \tilde I,\\
  \label{eq:v_tilde}
  \pd_\tau \tilde v &= d \SLT \tilde v + \gamma\Big(-\left(a_1 + (a_3 -
    a_1)\frac{\tilde u}{a_2 + \tilde u} \right)\tilde v + a_4
  \frac{\tilde u}{a_5 +\tilde u} + \tilde q\Big)\quad \text{on} \quad \tilde
  \Gamma \times \tilde I
\end{align}
with the flux condition
\begin{equation*}
  -\tilde D \nabla_\xi \tilde V \cdot \tilde \nu = \gamma \tilde q \quad
  \text{on} \quad \tilde \Gamma \times \tilde I   
\end{equation*}
with
\begin{equation}
  \label{eq:q_tilde}
  \tilde q = 
  a_6 \tilde V (1 - (\tilde u + \tilde v))_+ - a_{-6} \tilde v.
\end{equation}

\subsubsection{Reduction}
We further reduce the non-dimensional model of the previous section. Our reduction is motivated by the observation that the cytosolic diffusion coefficient is much larger than that of the lateral diffusion on the membrane \cite{PBLH04}. We thus assume $\tilde V$ to be spatially constant, i.e. $\tilde V =
\tilde V(\tau)$ depends only on time but not on the $\xi$ variable. If we also assume that the initial concentration of cytosolic GTPase is spatially homogeneous the concentration is then for positive times determined by GTPase conservation,
\begin{equation}
  \label{eq:V_tilde2}
  \tilde V(\tau) = \bar{V}_0 - c \int_{\tilde \Gamma} (\tilde u + \tilde v)(\xi,\tau) \dgxi
\end{equation}
where $c := |\tilde B|^{-1}$ and where $\bar{V}_0$ is given by the initial conditions,
\begin{gather*}
	\bar{V}_0 \,=\, \tilde{V}(0) + c \int_{\tilde \Gamma} (\tilde u + \tilde v)(0,\xi) \dgxi.
\end{gather*}
In particular $\bar{V}_0=\tilde{V}(0)$ if initially no membrane-bound GTPase was present.
In the following, we then consider the system
\eqref{eq:u_tilde}--\eqref{eq:v_tilde} of reaction diffusion equations
on $\tilde \Gamma \times \tilde I$ including the flux
\eqref{eq:q_tilde} and the conservation law
\eqref{eq:V_tilde2}.

The fully coupled system converges in the limit $D\to\infty$ to this
reduced model. However, no estimates for the difference between
solutions to the respective models are at present available.  The
ratio between cytosolic and lateral membrane diffusion reported in the
literature \cite{PBLH04} is of order $10^2$. Numerical experiments for
the full system with diffusion coefficients $d=D=10^3$ showed
qualitative agreement with the reduction.
\section{Turing pattern}\label{sec:turing}
In this section we investigate the stability properties and the possibility of Turing-type pattern formation for the dimensionless reduced model derived above. In the following we drop all tildes and denote the space and time variables by $x$ and $t$ respectively. This yields the system
\begin{align}
	\partial_t u \,&=\, \Delta_\Gamma u + \gamma f(u,v), \label{eq:u}\\
	\partial_t v \,&=\, d\Delta_\Gamma v + \gamma \left(-f(u,v) + q(u+v,v,V[u+v])\right) \label{eq:v}
\end{align}
where
\begin{align}
	f(u,v)\,&=\, \left(a_1 + (a_3 - a_1)\frac{u}{a_2 + u} \right) v - 
	a_4\frac{u}{a_5 +u},\\
	q(u+v,v,V)\,&=\, a_6 V (1 - (u + v))_+ - a_{-6} v,
\end{align}
and where $V[u+v]$ is the \emph{non-local} functional
\begin{gather}
	V[u+v]\,=\, V_0 - c \int_{\Gamma} (u + v) \dg, \label{eq:V}
\end{gather}
with $V_0$ given. For convenience we also define
\begin{gather*}
	g(u,v)\,=\, -f(u,v) + q(u+v,v,V[u+v]).
\end{gather*}
System \eqref{eq:u}, \eqref{eq:v} has to be solved on $\Gamma\times I$. Our particular interest is to understand the effect of the non-local term in system \eqref{eq:u}, \eqref{eq:v} on the stability properties. We can not use a predefined set of `realistic' parameter values: first our model is very general and applies to several specific cases with different set of parameters; second kinetic rates etc. are difficult to obtain experimentally. We therefore rather investigate whether \emph{in principle}, i.e. for \emph{some} parameter values, our model allows for stationary states, whether stationary states of substrate--depletion type exists, and whether Turing type instabilities are possible. It is quite difficult to guess parameters that allow for example for Turing pattern formation. We therefore derive conditions for the parameters that are sufficient to ensure certain behavior, in particular showing that the \emph{Turing space} is not empty. With such a set of parameters identified it is possible to explore the boundaries of the Turing space and then compare whether the parameter ranges are reasonably close to available estimates for `realistic' values.

To start with our stability analysis we first observe that spatially homogeneous solutions of \eqref{eq:u}, \eqref{eq:v} satisfy the ODE system
\begin{align}
	\partial_t u \,&=\, \gamma f(u,v), \label{eq:ode-u}\\
	\partial_t v \,&=\, \gamma \left(-f(u,v) + q_0(u+v,v)\right) \label{eq:ode-v}
\end{align}
where
\begin{gather}
	q_0(u+v,v)\,=\, a_6 \left(V_0 -c|\Gamma|(u+v)\right)(1 - (u + v))_+ - a_{-6}v.
\end{gather}
We set $g_0(u,v)\,=\, -f(u,v) + q_0(u+v,v)$. The set of values for $u,v$ described by
\begin{gather}
	\Aset \,:=\,  \big\{ u,v \,\geq\, 0 \,: \, u+v \leq \min\{1,m\}\big\},\quad m\,:=\,\frac{V_0}{c|\Gamma|}
\end{gather}
is an invariant region for \eqref{eq:ode-u}, \eqref{eq:ode-v}, \emph{i.e.} if the initial data are in this set the solution does not leave it. In fact we observe that at the boundaries of $\Aset$ we obtain
\begin{gather*}
  f(0,v)\,\geq\, 0,\quad g_0(u,0)\,\geq 0 \quad \text{ for all }
  0\,\leq\,u,v\,\leq \min\{1,m\},\\ 
  f(u,v)+g_0(u,v) \,\leq\,0\quad\text{ for all } u,v\geq 0,\,
  u+v=\min\{1,m\}. 
\end{gather*}
By these inequalities the conclusion follows. 

Under suitable conditions on the data we next show the existence of a stationary spatially homogeneous state $(u_*,v_*)\in\Aset$ for \eqref{eq:u}, \eqref{eq:v}. This means that $(u_*,v_*)$ has to satisfy
\begin{align}
	0\,&=\, f(u_*,v_*), \label{eq:f=0}\\
	0\,&=\, g_0(u_*,v_*). \label{eq:g=0}
\end{align}
The first equation is satisfied if and only if
\begin{gather}
	v\,= v[u]\,:=\, \frac{a_4 u(a_2+u)}{(a_5+u)(a_1a_2+a_3u)}. \label{eq:def-vu}
\end{gather}
We compute 
\begin{gather}\label{eq:v-prime}
	v'[u]=0 \quad\Leftrightarrow\quad \big(a_1a_2+a_3(a_5-a_2)\big)u^2 + 2a_1a_2a_5u +a_1a_2^2a_5\,=\,0.
\end{gather}
If we assume
\begin{align}
	a_2\,&> a_5, \label{eq:ass1}\\
	2 a_1 a_2 \,&<\, a_3(a_2-a_5) \label{eq:ass2}
\end{align}
we find that $v[\cdot]$ has a unique positive stationary point $u_0$,
\begin{gather*}
	u_0\,=\, \frac{a_1 a_2 a_5}{a_3(a_2-a_5)-a_1a_2} + a_2\sqrt{a_1a_5}\frac{\sqrt{(a_3-a_1)(a_2-a_5)}}{a_3(a_2-a_5)-a_1a_2}.
\end{gather*}
In particular we have
\begin{gather}
	v'[u]\,<\, 0\quad\text{ for all }u>u_0. \label{eq:est-v-prime}
\end{gather}
By \eqref{eq:ass1}, \eqref{eq:ass2} we estimate
\begin{align}
	u_0\,&\leq\, \sqrt{a_1a_5}a_2 \Big( \frac{2\sqrt{a_1a_2}}{a_3(a_2-a_5)} + 2\frac{1}{\sqrt{a_3(a_2-a_5)}}\Big)\label{eq:est-u0}\\
	&\,\leq\, 4\frac{\sqrt{a_1a_5}a_2}{\sqrt{a_3(a_2-a_5)}}. \notag
\end{align}
For the maximum value $v_0:=v[u_0]$ we obtain
\begin{gather*}
	v_0\,=\, a_4\frac{a_2+2 u_0}{a_1a_2+a_3(a_5+2u_0)}.
\end{gather*}
Using \eqref{eq:ass1} and \eqref{eq:ass2} a short calculation yields
\begin{align}
	v_0 \,&\leq\, \frac{a_2a_4}{a_1a_2 +a_3a_5}\,\leq\, \frac{a_2a_4}{a_3a_5}, \label{eq:upper-v0}\\
	v_0\,&\geq\, \frac{a_4a_5}{a_1a_2+a_3a_5}\,\geq\, \frac{a_4a_5}{a_3a_2}. \label{eq:lower-v0}
\end{align}
If we then choose
\begin{align}
	\frac{a_2a_4}{a_3a_5}\,&<\, \frac{1}{4}\min\{m,1\}, \label{eq:ass3} \\
	a_1\,&<\, 2^{-8}\frac{a_3^2(a_2-a_5)^2}{a_2^2a_5}(\min\{m,1\})^2 \label{eq:ass4}
\end{align}
we deduce by \eqref{eq:est-u0}, \eqref{eq:upper-v0} that $u_0+v_0<\frac{1}{2}\min\{m,1\}$. 

In order to satisfy \eqref{eq:g=0} we need to find $u>0$ with $u+v[u]<\min\{m,1\}$ such that
\begin{gather*}
	0\,=\, a_6 \left(V_0 -c|\Gamma|(u+v[u])\right)(1 - (u + v[u])) - a_{-6}v[u]\,=:\,\Phi(u).
\end{gather*}
We evaluate
\begin{gather*}
	\Phi(u_0)\,>\, \frac{a_6}{4}V_0 - a_{-6}\frac{a_4a_5}{a_3a_2}
\end{gather*}
and assuming
\begin{gather}
	\frac{a_{-6}}{a_{6}}\,\leq\,\frac{V_0}{4}\frac{a_3a_2}{a_4a_5} \label{eq:ass5}
\end{gather}
we see that $\Phi(u_0)>0$. On the other hand there exists $u_1>u_0$ such that $u_1+v[u_1]=\min\{m,1\}$ and we observe that
\begin{gather*}
	\Phi(u_1)\,<\, 0.
\end{gather*}
Since $\Phi$ is continuous we obtain from the intermediate-value
Theorem that there exists $u_0<u_*<u_1$ such that $\Phi(u_*)=0$. But
this implies that $(u_*,v_*)\in\Aset$, $v_* = v[u_*]$, is a stationary point of \eqref{eq:ode-u},\eqref{eq:ode-v}.
In summary we have proved the following Proposition.
\begin{proposition}\label{prop:stat-sol}
Assume that the conditions
\begin{align}
	a_2\,&>\, a_5, \label{cdt:1}\\
	4a_2a_4\,&<\, a_3a_5\min\{m,1\}, \label{cdt:2}\\
	4a_4a_5a_{-6}\,&<\, V_0 a_2a_3a_6, \label{cdt:3}\\
	a_1\,&<\, \min\Big\{ \frac{a_3(a_2-a_5)}{2a_2}\,,\,2^{-8}\frac{a_3^2(a_2-a_5)^2}{a_2^2a_5}(\min\{m,1\})^2 \Big\}\label{cdt:4}
\end{align}
are satisfied. Then there exists a stationary spatially homogeneous solution $(u_*,v_*)\in\Aset$ of \eqref{eq:u}, \eqref{eq:v}.
\end{proposition}
The stationary point $(u_*,v_*)$ is under suitable assumptions on the data linearly stable against spatially homogeneous perturbations. For a brief summary of the classical stability analysis and of the Turing mechanism for two-variable reaction--diffusion systems we refer to Appendix \ref{app:turing}.
\begin{proposition}\label{prop:stab}
Assume that \eqref{cdt:1}-\eqref{cdt:4} hold and that moreover
\begin{align}
	2a_4(a_2-a_5)\,&<\, a_3a_5^2, \label{cdt:5}\\
	a_{-6}\,&<\, a_6 c|\Gamma||1-m| \label{cdt:6}
\end{align}
are satisfied. Then $(u_*,v_*)$ is a stable stationary point of \eqref{eq:ode-u}, \eqref{eq:ode-v}. This system is in $(u_*,v_*)$ of activator--substrate-depletion type, where $u$ acts as an activator and $v$ as substrate.
\end{proposition}
\begin{proof}
We show that the stability conditions \eqref{eq:rd-Tu1}, \eqref{eq:rd-Tu2} are satisfied. 
We first observe that since $a_1<\frac{a_3}{2}$ by \eqref{cdt:4}
\begin{align}
	\partial_v f(u,v)\,&=\, a_1 + (a_3 - a_1)\frac{u}{a_2 + u} \label{eq:pd-f-v}\\
	&>\, 0\quad\text{ for all }u>0. \label{eq:del-v-f}
\end{align}
For the function $v[\cdot]$ defined in \eqref{eq:def-vu} we have $f(u,v[u])=0$. This yields
\begin{gather*}
	0\,=\,\partial_u f(u,v[u])\,=\, (\partial_u f)(u,v[u]) + (\partial_v f)(u,v[u])v'[u].
\end{gather*}
Since $u_*>u_0$ we deduce from \eqref{eq:est-v-prime} that $v'[u_*]<0$ and obtain
\begin{gather}
	\partial_u f(u_*,v_*)\,=\, - \partial_vf(u_*,v_*)v'[u_*]\,>\,0. \label{eq:del-u-f}
\end{gather}
Furthermore we have
\begin{align}
	\partial_u q_0(u,v)\,&=\, -a_6 c|\Gamma|\big(1+m-2(u+v)\big)\,<\,0,\label{eq:del-u-q0}\\
	\partial_v q_0(u,v)\,&=\, -a_6 c|\Gamma|\big(1+m-2(u+v)\big)-a_{-6}\,<\,0. \label{eq:del-v-q0}
\end{align}
By \eqref{eq:del-v-f}-\eqref{eq:del-v-q0} the stationary point $(u_*,v_*)$ is of activator--substrate-depletion type.
To check the criteria for Turing type instabilities we need to estimate combinations of derivatives. We first compute
\begin{gather*}
	\partial_uf \,=\, \frac{a_2(a_3-a_1)}{(a_2+u)^2}v - \frac{a_4a_5}{(a_5+u)^2}.
\end{gather*}
Evaluating this expression at $(u,v)=(u_*,v_*)$ and using \eqref{eq:def-vu} we deduce
\begin{align}
	\partial_uf \,&=\, \frac{a_2(a_3-a_1)a_4 u}{(a_2+u)(a_5+u)(a_1a_2+a_3u)}- \frac{a_4a_5}{(a_5+u)^2}\notag\\
		&\leq\, \frac{a_2 a_3a_4}{(a_2+u)(a_5+u)a_3}- \frac{a_4a_5}{(a_5+u)^2} \notag\\
		&=\, \frac{a_4(a_2 -a_5)u}{(a_2+u)(a_5+u)^2}. \label{eq:pd-f-u}
\end{align}
We thus obtain in $(u,v)=(u_*,v_*)$ that
\begin{align}
	\partial_uf + \partial_v g_0 \,\leq\,&   \frac{a_4(a_2 -a_5)u}{(a_2+u)(a_5+u)^2} -\Big(a_1 + (a_3 - a_1)\frac{u}{a_2 + u}\Big) \notag\\
	&-a_6 c|\Gamma|\big(1+m-2(u+v)\big)-a_{-6}\notag\\
	\leq\,& \frac{a_4(a_2 -a_5)u}{(a_2+u)(a_5+u)^2} -a_3\frac{u}{a_2 + u}\,
	<\, 0 \label{eq:turing-cdt1}
\end{align}
by \eqref{cdt:5}. This verifies condition \eqref{eq:rd-Tu1}. We next estimate in $(u,v)=(u_*,v_*)$ 
\begin{align}
	&\partial_uf\partial_vg_0 - \partial_vf\partial_ug_0 \notag\\
	=\, &\partial_uf(-\partial_vf + \partial_vq_0) - \partial_vf(-\partial_uf + \partial_uq_0)\notag\\
	=\, & \partial_uq_0(\partial_uf-\partial_vf)-a_{-6}\partial_uf\notag\\
	\geq\,& -a_6 c|\Gamma|\big(1+m-2(u+v)\big)\Big(\frac{a_4(a_2 -a_5)u}{(a_2+u)(a_5+u)^2} - a_3\frac{u}{a_2 + u}\Big)  - a_{-6}\frac{a_4(a_2 -a_5)u}{(a_2+u)(a_5+u)^2}\notag\\
	=\,& -\frac{u}{(a_2+u)(a_5+u)^2}\Big(a_6 c|\Gamma|\big(1+m-2(u+v)\big)\big(a_4(a_2 -a_5)-a_3(a_5+u)^2\big)+a_{-6}a_4(a_2 -a_5)\Big) \label{eq:Tu-hilf}
\end{align}
By \eqref{cdt:5} and since $1+m-2(u+v)\geq |1-m|$ the term in the brackets in the last line is estimated by
\begin{align*}
	&\,a_6 c|\Gamma|\big(1+m-2(u+v)\big)\big(a_4(a_2 -a_5)-a_3(a_5+u)^2\big)+a_{-6}a_4(a_2 -a_5)\\
	\leq\,& -a_6 c|\Gamma||1-m|a_4(a_2 -a_5)+a_{-6}a_4(a_2 -a_5)\,<\, 0,
\end{align*}
where the last estimate follows from \eqref{cdt:6}. Together with \eqref{eq:Tu-hilf} the last equation implies
\begin{gather*}
	\partial_uf\partial_vg_0 - \partial_vf\partial_ug_0 \,>\,0,
\end{gather*}
and therefore \eqref{eq:rd-Tu2} holds. Thus the ODE system is linearly stable.
\end{proof}
We next evaluate the response of the full reaction--diffusion system to perturbations of the spatially homogeneous solution $(u_*,v_*)$ in direction of arbitrary smooth functions $\varphi,\psi:\Gamma\times(0,T)\to\R$. In particular we have to linearize the non-local functional $V=V[u+v]$ that occurs in the source term $q$ in \eqref{eq:v}. With this aim we consider a variation $(u_s,v_s)$ of $(u_*,v_*)$ in direction of $(\varphi,\psi)$,
\begin{gather*}
	u_s,v_s \,:\, \Gamma\times (0,T)\,\to\,\R,\quad  s\in (-1,1),\\
	u_s\big|_{s=0}\,=\, u_*,\quad v_s\big|_{s=0}\,=\, v_*,\quad \frac{\partial}{\partial s}\Big|_{s=0}u_s\,=\, \varphi, \quad \frac{\partial}{\partial s}\Big|_{s=0}v_s\,=\,\psi.
\end{gather*}
The corresponding linearization of $V$ is then given by
\begin{align}
	\frac{d}{ds}\big|_{s=0} V[u_s+v_s]\,=\, -c\frac{d}{ds}\big|_{s=0} \int_\Gamma (u+v)\,\dgxi\,=\, -c\int_{\Gamma} (\varphi + \psi) \dgxi. \label{eq:lin-V}
\end{align}
For the linearization of \eqref{eq:u}, \eqref{eq:v} we therefore obtain
\begin{align}
	\partial_t \varphi \,&=\, \Delta_\Gamma \varphi + \gamma \partial_u f(u_*,v_*)\varphi + \gamma\partial_vf(u_*,v_*)\psi, \label{eq:u-lin}\\
	\partial_t \psi \,&=\, d\Delta_\Gamma \psi + \gamma \big(-\partial_u f(u_*,v_*)\varphi -\partial_v f(u_*,v_*) \psi +\partial_1q(u_*+v_*,v_*,V_*)(\varphi+\psi)\big)+ \label{eq:v-lin}\\
	&\qquad \qquad + \gamma\Big(\partial_2q(u_*+v_*,v_*,V_*)\psi- c \partial_3q(u_*+v_*,v_*,V_*)\int_{\Gamma} (\varphi + \psi) \dgxi\Big). \notag
\end{align}
We next decompose the direction of perturbation $(\varphi,\psi)$ in $L^2(\Gamma)$ in a part that is spatially homogeneous and a part that is orthogonal to the constants. Since spatially homogeneous perturbations have already been analyzed in Proposition \ref{prop:stab} it suffices to assume that we have
\begin{gather*}
	\int_\Gamma \varphi \dg\,=\, \int_\Gamma \psi\dg \,=\, 0.
\end{gather*}
Equation \eqref{eq:lin-V} shows that for such directions $V$ is unchanged to first order. Therefore, with respect to variations in direction of functions that are orthogonal to the constants, the linearizations of \eqref{eq:u}, \eqref{eq:v} coincide with that of the system
\begin{align}
	\partial_t u \,&=\, \Delta_\Gamma u + \gamma f(u,v), \label{eq:u-het}\\
	\partial_t v \,&=\, d\Delta_\Gamma v + \gamma \left(-f(u,v) + q_1(u+v,v)\right) \label{eq:v-het}
\end{align}
in $(u_*,v_*)$, where
\begin{align}
	q_1(u+v,v)\,&=\, q(u+v,v,V_*)\,=\, a_6 V_* (1 - (u + v)) - a_{-6} v, \label{eq:def-q1}
\end{align}
$V_*=V[u_*+v_*]$
For convenience we define
\begin{gather*}
	g_1(u,v)\,=\, -f(u,v) + q_1(u+v,v).
\end{gather*}
Thus we see that the non-local term in the full system \eqref{eq:u}, \eqref{eq:v} leads to a difference in the linearization with respect to homogeneous or heterogeneous perturbations, that can be understood as a change (from $q_0$ to $q_1$) in the source term. Below we show that a range of parameters exists where $(u_*,v_*)$ is an unstable stationary state. We first need some estimates on $u_*,v_*$ to prepare the proof.
\begin{lemma}\label{lem:estimates}
Assume \eqref{cdt:1}, \eqref{cdt:2}, \eqref{cdt:6}, and
\begin{gather}
	a_1a_2\,<\, \frac{a_3}{1+a_3^2}. \label{cdt:7}
\end{gather}
Then
\begin{align}
	v_*\,&<\, \frac{a_4 a_2}{a_3a_5}, \label{eq:upper-v-star}\\
	u_*\,&>\, \frac{1}{2}\min\{1,m\}, \label{eq:lower-u-star}\\
	v_*\,&>\, \frac{a_4(a_2+1)(a_3^2+1)\min\{1,m\}}{2(a_5+1)(a_3^2+2)a_3}, \label{eq:lower-v-star}
\end{align}
holds.
\end{lemma}
\begin{proof}
By \eqref{eq:upper-v0} and $v_*<v_0$ we deduce \eqref{eq:upper-v-star}. From $q_0(u_*,v_*)=0$ we obtain that $(u_*,v_*)$ is a solution of
\begin{gather}
	0\,=\, a_6c|\Gamma|\big(m-(u+v)\big)\big(1-(u+v)\big) - a_{-6}v. \label{eq:def-uv}
\end{gather}
If we denote by $u[v]\in (0,\min\{1,m\})$ the solution of \eqref{eq:def-uv} for given $0<v<\min\{1,m\}$ we observe that $u$ is decreasing in $v$. By \eqref{cdt:2},\eqref{eq:upper-v-star} we have $v_*<v_1:=\frac{1}{4}\min\{1,m\}$ and therefore  $u_*> u_1$, $u_1:=u[v_1]$. By \eqref{eq:def-uv} and \eqref{cdt:6} we get
\begin{align*}
	u_1\,&=\, \frac{1}{2}(m+1-2v_1) - \sqrt{\frac{1}{4}(m+1-2v_1)^2 -(m-v_1)(1-v_1)+ \frac{a_{-6}v_1}{a_6c|\Gamma|}}\\
	&\geq\, \frac{1}{4}(2\max\{1,m\}+\min\{1,m\}) - \sqrt{\frac{1}{4}(m-1)^2 + \frac{1}{4}|1-m|\min\{1,m\}} \\
	&=\, \frac{1}{2}\min\{1,m\},
\end{align*}
which proves \eqref{eq:lower-u-star}.
Next we obtain from $f(u,v)=0$ for $(u,v)=(u_*,v_*)$ that
\begin{align*}
	v\,=\, \frac{a_4u(a_2+u)}{(a_5+u)(a_1a_2+a_3u)}\,\geq\, \frac{a_4(a_2+1)u}{(a_5+1)(a_1a_2+a_3u)}\,\geq\, \frac{a_4(a_2+1)(a_3^2+1)u}{(a_5+1)(a_3^2+2)a_3},
\end{align*}
where we have used \eqref{cdt:1} and \eqref{cdt:7}. By \eqref{eq:lower-u-star} this yields \eqref{eq:lower-v-star}.
\end{proof}
\begin{theorem}\label{thm:turing}
Let $(u_*,v_*)$ be  the stationary state found in Proposition \ref{prop:stat-sol} and let the parameters in \eqref{eq:u}, \eqref{eq:v} satisfy all the conditions \eqref{cdt:1}-\eqref{cdt:6}. If in addition
\begin{align}
	a_1\,&<\, \min\Big\{\frac{\min\{1,m\}a_3}{2a_2(a_3^2+1)}\,,\, \frac{\min\{1,m\}^2a_3(a_2-a_5)}{4a_2(a_2+1)^2(a_3^2+1)}\Big\}, \label{cdt:8}\\
	d\,&>\,	\frac{2\big(a_3(a_3^2+2)+(a_2+1)(a_3^2+1)(a_6V_0 +a_{-6})\big)(a_3^2+2)(a_5+1)^2}{\min\{1,m\}a_3^2a_4(a_2-a_5)(a_3^2+1)},\label{cdt:d1}\\
	d\,&>\,  \frac{4(a_3^2+2)^2(a_5+1)^2\big(a_3^2a_4(a_2-a_5)\min\{1,m\}+4(a_3^2+2)a_6V_0(a_2+1)(a_5+1)^2\big)}{a_3^3(a_3^2+1)a_4^2(a_2-a_5)^2\min\{1,m\}^2}\label{cdt:d2}
\end{align}
then there exists $\gamma>0$ such that $(u_*,v_*)$ is linearly unstable.
\end{theorem}
\begin{proof}
We show that the conditions \eqref{eq:rd-Tu3}, \eqref{eq:rd-Tu4} stated in Appendix \ref{app:turing} for the instability of \eqref{eq:u-het}, \eqref{eq:v-het} are satisfied. Let us start with \eqref{eq:rd-Tu3} by
estimating the different partial derivatives.
\begin{align*}
	\partial_uf(u_*,v_*)\,&=\, \frac{a_2(a_3-a_1)a_4u(a_5+u)-a_4a_5(a_1a_2+a_3u)(a_2+u)}{(a_2+u)(a_5+u)^2(a_1a_2+a_3u)}\\
	&=\, \frac{a_4\Big(a_3(a_2-a_5)u^2-a_1a_2(2a_5u+u^2+a_2a_5)\Big)}{(a_2+u)(a_5+u)^2(a_1a_2+a_3u)}.
\end{align*}
We next observe that by \eqref{cdt:1}, \eqref{cdt:8} and \eqref{eq:lower-u-star}
\begin{align}
	a_1a_2(2a_5u+u^2+a_2a_5)\,&\leq\, a_1a_2(1+a_2)^2\,\leq\, \frac{a_3}{1+a_3^2}(a_2-a_5)u^2,\notag\\
	a_1a_2 \,&\leq\, \frac{a_3}{1+a_3^2}u.\label{cdt:7a}
\end{align}
We therefore deduce that
\begin{align}
	\partial_uf(u_*,v_*)\,&\geq\, \frac{a_4(a_3-\frac{a_3}{1+a_3^2})(a_2-a_5)u^2}{(a_2+u)(a_5+u)^2(\frac{a_3}{1+a_3^2}+a_3)u}\notag\\
	&=\, \frac{a_4a_3^2(a_2-a_5)u}{(a_2+1)(a_5+1)^2(a_3^2+2)}\notag\\
	&\geq\, \frac{a_4a_3^2(a_2-a_5)\min\{1,m\}}{2(a_2+1)(a_5+1)^2(a_3^2+2)}.\label{eq:lower-dfu}
\end{align}
Next we estimate
\begin{align*}
	\partial_v g_1(u_*,v_*) \,&=\, -\Big(a_1 + (a_3 - a_1)\frac{u}{a_2 + u}\Big) -\Big(a_6 \big(V_0-c|\Gamma|(u+v)\big) +a_{-6}\Big)\\
	&\geq\, -\frac{a_1a_2+a_3u}{a_2+u} -\big(a_6V_0 +a_{-6}\big)\\
	&\geq\, -\frac{a_3(a_3^2+2)}{(a_2+1)(a_3^2+1)}-\big(a_6V_0 +a_{-6}\big),
\end{align*}	
where we have used \eqref{cdt:7a}.
Together with \eqref{eq:lower-dfu} we deduce from \eqref{cdt:d1} that \eqref{eq:rd-Tu3} holds.

Next we verify the condition \eqref{eq:rd-Tu4}, \emph{i.e.}
\begin{gather*}
	D\,:=\,(d\partial_uf+\partial_vg_1)^2 -4d(\partial_uf\partial_vg_1-\partial_vf\partial_ug_1)\,>\, 0. 
\end{gather*}
We compute 
\begin{gather*}
	\partial_uf\partial_vg_1-\partial_vf\partial_ug_1\,=\, \partial_uf\partial_vq_1 -\partial_vf\partial_uq_1
\end{gather*}
and obtain for the left hand side
\begin{align*}
	D\,&=\,d^2(\partial_uf)^2 +2d\Big(-\partial_uf(\partial_vf+\partial_vq_1)+2\partial_vf\partial_uq_1\Big) +(\partial_vg_1)^2\\
	&\geq\, d\Big[ d(\partial_uf)^2 -2\Big(\partial_uf\partial_vf-2\partial_vf\partial_uq_1\Big)\Big].
\end{align*}
Moreover
\begin{align*}
	\partial_vf\,&=\, \frac{a_1a_2+a_3u}{a_2+u}\,\leq\, \frac{a_3(a_3^2+2)u}{(a_2+u)(a_3^2+1)}\,\leq\, \frac{a_3(a_3^2+2)}{(a_2+1)(a_3^2+1)},\\
	-\partial_uq_1\,&=\,a_6V_*\,\leq\, a_6V_0.
\end{align*}
This yields
\begin{align*}
	D\,&\geq\, d(\partial_uf)^2\Big[ d - \frac{2a_3(a_3^2+2)}{(a_2+1)(a_3^2+1)(\partial_uf)} - \frac{4a_3(a_3^2+2)a_6V_0}{(a_2+1)(a_3^2+1)(\partial_uf)^2}\Big].
\end{align*}
We then deduce \eqref{eq:rd-Tu4} from \eqref{cdt:d2} and \eqref{eq:lower-dfu}. 
The conclusion now follows from  \cite{Jost07}.
\end{proof}
The above conditions ensure that perturbations with eigenvectors of the Laplace--Beltrami operator on $\Gamma$ corresponding to eigenvalues in a certain interval are unstable. As this interval scales linearly with $\gamma$ we obtain a range of values for $\gamma$ where a Turing instability exists. 

We conclude this section with some comments on the implications of the conditions derived above.
\begin{remark}
By Theorem \ref{thm:turing} parameters satisfying \eqref{cdt:1}-\eqref{cdt:6} and \eqref{cdt:8}-\eqref{cdt:d2} belong to the \emph{Turing space} where diffusive instabilities exist. Some of these conditions can be easily interpreted. The requirement $a_2>a_5$ concerns the Michaelis--Menten constants appearing in the catalyzed reactions: the affinity of GEF towards activated GTPase (forming the GEF--GTP-GTPase--effector complex) has to be higher than the affinity of GAP towards activated GTPase. Several conditions require $a_1$ to be (much) smaller than $a_3$, which means that activation by the GEF--GTP-GTPase--effector complex is stronger than activation by single GEF molecules, which in fact has been reported for example in the case of Rab5 GTPase \cite{LMRR01}. The conditions \eqref{cdt:d1}, \eqref{cdt:d2} for a Turing instability are 
most critical, as a substantially larger lateral diffusion coefficient for inactive GTPase compared to the lateral diffusion coefficient for active GTPase is required. We investigate below whether this condition is due to the particular choices of kinetic and sorption laws or rather a general feature of the kind of (reduced) model that we are considering. Finally, the condition on $\gamma$ requires a certain minimal size of the system to allow for a Turing instability. 
\end{remark}
\section{Stability analysis for equal lateral diffusion}\label{sec:no-turing}
As in most applications no substantial difference in the diffusion coefficients of the GDP-bound and GTP-bound GTPase is present we investigate in this section the possibility  of Turing pattern in \eqref{eq:u}, \eqref{eq:v} for the case $d=1$ of equal lateral diffusion. The non-locality of our model -- the remnant of the full 2D-3D coupling in our reduction -- changes the classical stability analysis. Therefore, in contrast to the classical case, Turing pattern for $d=1$ might become possible. However, we show here that in our simple model this is not the case.

The set-up in this section is as follows. We assume a system of the general form
\begin{align}
	\partial_t u \,&=\, \Delta_\Gamma u + \gamma f(u,v), \label{eq:u-gen}\\
	\partial_t v \,&=\, \Delta_\Gamma v + \gamma \left(-f(u,v) + q(u+v,v,V[u+v])\right) \label{eq:v-gen}
\end{align}
where $u,v$ denote GTP-bound and GDP-bound GTPase concentrations, respectively, and where $V[u+v]$ represents the cytosolic (inactive) GTPase concentration that is given by the mass conservation condition \eqref{eq:V}. The nonlinear terms $f,q$ account for the activation/deactivation processes and from the attachment/detachment of GTPase at the membrane. For $q$ we assume that
\begin{gather}
	\partial_1 q\,\leq\, 0,\quad \partial_2q\,\leq\, 0, \quad \partial_3q\,\geq\, 0,\label{cdt:q}
\end{gather}
which are natural condition with respect to the interpretation of $q$ as the flux induced by ad- and desorption of GTPase at the membrane.
As before, the system \eqref{eq:u}, \eqref{eq:v} has to be solved on $\Gamma\times I$.

We assume a spatially homogeneous stationary point $(u_*,v_*)$ of the ODE reduction of \eqref{eq:u-gen}, \eqref{eq:v-gen},
\begin{align}
	\partial_t u \,&=\, \gamma f(u,v), \label{eq:u-gen-ode}\\
	\partial_t v \,&=\, \gamma \left(-f(u,v) + q_0(u+v,v)\right) \label{eq:v-gen-ode}
\end{align}
where
\begin{gather*}
	q_0(u+v,v)\,=\, q(u+v,v,V_0(u+v)),\quad V_0(u+v)\,=\,V_0 -c|\Gamma|(u+v).
\end{gather*}
We set as before
\begin{gather*}
	g(u,v)\,:=\, -f(u,v) + q(u+v,v,V[u+v]),\qquad g_0(u,v)\,=\, -f(u,v) + q_0(u+v,v).
\end{gather*}
The main result of this section is the following theorem.
\begin{theorem}\label{thm:no-turing}
Assume that $(u_*,v_*)$ is of activator--substrate depletion type, that is
\begin{align}
	\partial_uf(u_*,v_*)\,&>0\, & \partial_vf(u_*,v_*)\,&>\, 0, \label{eq:acsub1}\\
	\partial_ug_0(u_*,v_*)\,&<0\, & \partial_vg_0(u_*,v_*)\,&>\, 0. \label{eq:acsub2}
\end{align}
Then no Turing type instability of \eqref{eq:u-gen}, \eqref{eq:v-gen} exists.
\end{theorem}
\begin{proof}
The conditions \eqref{eq:rd-Tu1}, \eqref{eq:rd-Tu2} to ensure  that $(u_*,v_*)$ is a stable stationary point of \eqref{eq:u-gen-ode}, \eqref{eq:v-gen-ode} are
\begin{align}
	0\,&>\, \partial_u f +\partial_v g_0\,=\, \partial_u f -\partial_v f  + \partial_1 q + \partial_2 q +\partial_3qV_0', \label{eq:Tu-gen-1}\\
	0\,&<\, \partial_u f\cdot\partial_vg_0 -\partial_vf\cdot\partial_ug_0\,
	=\, \partial_u f(\partial_1 q+\partial_2q+\partial_3qV_0') - \partial_v f (\partial_1 q+\partial_3qV_0'). \label{eq:Tu-gen-2}
\end{align}
This corresponds to the stability of \eqref{eq:u-gen}, \eqref{eq:v-gen} in $(u_*,v_*)$ with respect to spatially homogeneous perturbations. Therefore, for the instability with respect to general perturbations it is sufficient to consider perturbations in direction of functions perpendicular to the constants. As above we observe that such perturbations leave $V[u,v]$ unchanged. The respective linearization corresponds to that of the system
\begin{align}
	\partial_t u \,&=\, \Delta_\Gamma u + \gamma f(u,v), \label{eq:u-gen-het}\\
	\partial_t v \,&=\, d\Delta_\Gamma v + \gamma \left(-f(u,v) + q_1(u+v,v)\right) \label{eq:v-gen-het}
\end{align}
in $(u_*,v_*)$, where
\begin{align}
	q_1(u+v,v)\,&=\, q(u+v,v,V_*),\quad V_*=V[u_*+v_*]=V_0(u_*,v_*). \label{eq:def-q1-gen}
\end{align}
We set as before $g_1(u,v)\,=\, -f(u,v) + q_1(u+v,v)$.
Then the conditions \eqref{eq:rd-Tu3}, \eqref{eq:rd-Tu4} for the instability of \eqref{eq:u-gen-het}, \eqref{eq:v-gen-het} with respect to spatially heterogeneous perturbations yield
\begin{align}
	0\,&<\, \partial_u f+ \partial_v g_1
	\,=\, \partial_u f-\partial_vf + \partial_1 q +\partial_2 q, \label{eq:Tu-gen-3}\\
	0\,&<\, (\partial_u f +\partial_v g_1)^2 - 4 (\partial_uf\cdot\partial_vg_1-\partial_vf\cdot\partial_ug_1) \notag\\
	&=\, (\partial_u f -\partial_v f -\partial_1 q +\partial_2 q)^2 -4\partial_uf \cdot\partial_2 q + 4\partial_1q\cdot\partial_2q.\label{eq:Tu-gen-4}
\end{align}
By our assumptions \eqref{eq:acsub1}, \eqref{cdt:q} we observe that the last condition is automatically satisfied. On the other hand, we obtain from \eqref{eq:Tu-gen-2} that
\begin{gather}
	0\,<\, (\partial_u f-\partial_vf)(\partial_1 q+\partial_3qV_0') 	
		+\partial_u f \partial_2q. \label{eq:Tu-gen-2b}
\end{gather}
By \eqref{eq:Tu-gen-3} and \eqref{cdt:q}  we deduce that $\partial_u f-\partial_vf\,=\,-(\partial_1 q +\partial_2 q)\,>\,0$. Using again \eqref{cdt:q} and using that $V_0'\leq 0$ this yields that the first term on the right-hand side in \eqref{eq:Tu-gen-2b} is nonpositive. But we also find by \eqref{eq:acsub1}, \eqref{cdt:q} that the second term is nonpositive. This is a contradiction. Therefore the system has no Turing-type instabilities.
\end{proof}
\section{Numerical Approach}\label{sec:numerics}

We use a finite element discretization similar to the one described in
\cite{LaVo10}. It is implemented in the adaptive finite element
toolbox AMDiS \cite{VeVo07}. 

\subsection{Discretization}

Following the surface finite element method described in
\cite{DzEl07}, we choose a triangulated discrete approximation
$\Gamma_h$ of the membrane $\Gamma$ and a triangulation $\mathcal
T_h$. We split the time interval $[0,T]$ by discrete time instants
$t_0 < t_1 < \dots < t_{M}$, from which one gets the time steps
$\Delta t_m := t_{m + 1} - t_m$, $m = 0, 1, \dots, M-1$. Given initial
conditions $u(\cdot,0) = u_0$, $v(\cdot,0) = v_0$ with $u_0,v_0 \in
H^1(\Gamma_h)$ and time discrete solutions $u^{(m)}, v^{(m)} \in
H^1(\Gamma_h)$, $m = 1, \dots, M$, we linearize all nonlinear terms
\begin{align*}
  f(u^{(m+1)}, v^{(m+1)}) \approx f(u^{(m)}, v^{(m)}) + \nabla
  f(u^{(m)}, v^{(m)}) \cdot 
  \begin{pmatrix}
    u^{(m+1)} - u^{(m)}\\
    v^{(m+1)} - v^{(m)}
  \end{pmatrix}
\end{align*}
and 
\begin{align*}
  q(u^{(m+1)}, v^{(m+1)}, V^{(m+1)}) \approx q(u^{(m)}, v^{(m)},
  V^{(m)}) \\ + \nabla_{(u,v)}
  q(u^{(m)}, v^{(m)}, V^{(m)}) \cdot 
  \begin{pmatrix}
    u^{(m+1)} - u^{(m)}\\
    v^{(m+1)} - v^{(m)}
  \end{pmatrix}.
\end{align*}
We introduce test functions $\eta^u, \eta^v \in H^1(\Gamma_h)$ and end
up with a weak formulation and semi-implicit time discretization
for $u^{(m+1)}, v^{(m+1)} \in H^1(\Gamma_h)$ of \eqref{eq:u}, \eqref{eq:v}
\begin{align*}
  \frac1{\Delta t_m}&\int_{\Gamma_h}u^{(m+1)}\eta^u 
  + \int_{\Gamma_h}\scpgh{\SG u^{(m+1)}}{\SG \eta^u }
  - \gamma\int_{\Gamma_h} \nabla f(u^{(m)}, v^{(m)}) \cdot
  \begin{pmatrix}
    u^{(m+1)}\\
    v^{(m+1)}
  \end{pmatrix}
  \eta^u\\
  &= \frac1{\Delta t_m}\int_{\Gamma_h}u^{(m)}\eta^u 
  + \int_{\Gamma_h} F_e(u^{(m)}, v^{(m)}) \eta^u
  \quad \forall \eta^u \in H^1(\Gamma_h)\\
  \frac1{\Delta t_m}&\int_{\Gamma_h}v^{(m+1)}\eta^v 
  + d\int_{\Gamma_h}\scpgh{\SG v^{(m+1)}}{\SG \eta^v }
  + \gamma\int_{\Gamma_h} \nabla f(u^{(m)}, v^{(m)}) \cdot
  \begin{pmatrix}
    u^{(m+1)}\\
    v^{(m+1)}
  \end{pmatrix}
  \eta^u\\
  &+ \gamma\int_{\Gamma_h} \nabla q(u^{(m)}, v^{(m)}, V^{(m)}) \cdot
  \begin{pmatrix}
    u^{(m+1)}\\
    v^{(m+1)}
  \end{pmatrix}
  \eta^u\\
  &= \frac1{\Delta t_m}\int_{\Gamma_h}v^{(m)}\eta^v 
  - \int_{\Gamma_h} F_e(u^{(m)}, v^{(m)}) \eta^v
  + \int_{\Gamma_h} Q_e(u^{(m)}, v^{(m)}, V^{(m)}) \eta^v
  \quad \forall \eta^v \in H^1(\Gamma_h),
\end{align*}
where
\begin{align*}
  F_e(u^{(m)}, v^{(m)}) &:= \gamma f(u^{(m)}, v^{(m)}) - \gamma \nabla
  f(u^{(m)}, v^{(m)}) \cdot
  \begin{pmatrix}
    u^{(m)}\\
    v^{(m)}
  \end{pmatrix}\\
  Q_e(u^{(m)}, v^{(m)}, V^{(m)}) &:= \gamma q(u^{(m)}, v^{(m)},V^{(m)}) -
  \gamma\nabla_{(u,v)} q(u^{(m)}, v^{(m)}, V^{(m)}) \cdot
  \begin{pmatrix}
    u^{(m)}\\
    v^{(m)}
  \end{pmatrix}.
\end{align*}
Furthermore, the non-local relation for $V^{(m)}$ is treated
explicitly by
\begin{equation*}
  V_h^{(m+1)} = V_0 - \frac1{|B_h|}\int_{\Gamma_h}(u^{(m)}+v^{(m)})
\end{equation*}
for given $V_0$ and the inner $B_h$ of $\Gamma_h$. To discretize in
space, let $\mathbb V_h$ the finite element space of globally
continuous, piecewise linear elements. In addition, with $(\psi_i)_i$
the standard nodal basis of $\mathbb V_h$ and $u_h^{(m+1)},
v_h^{(m+1)} \in \mathbb V_h$ we write $u_h^{(m+1)} = \sum \limits_i
U_i^{(m+1)}\psi_i$ and $v_h^{(m+1)} = \sum \limits_i V_i^{(m+1)}
\psi_i$ with $U_i^{(m+1)}, V_i^{(m+1)} \in \R$. Furthermore, we define
$\matr{U}^{(m+1)}=(U_i^{(m+1)})_i$ and
$\matr{V}^{(m+1)}=(V_i^{(m+1)})_i$. This leads to the linear system of
equations 
\begin{multline*}
  \begin{pmatrix}
    \frac1{\taum}\matr{M} + \matr{A} - \matr{F_u^{\text{impl}}} 
    & \frac{1}{\taum} \matr{M} - \matr{F_v^{\text{impl}}} \\
    \matr{F_u^{\text{impl}}} - \matr{Q_u^{\text{impl}}} 
    &   \frac{1}{\taum} \matr{M}
    + d \matr{A}  + \matr{F_v^{\text{impl}}}  - \matr{Q_v^{\text{impl}}}
  \end{pmatrix}
  \begin{pmatrix}
    \matr{U}^{(m+1)} \\
    \matr{V}^{(m+1)}  
  \end{pmatrix}\\
  =
  \begin{pmatrix}
    \frac{1}{\taum} \matr{M} \matr{U}^{(m)} + \matr{F^{\text{expl}}} \\
    \frac{1}{\taum} \matr{M}\matr{V}^{(m)}
    - \matr{F^{\text{expl}}} + \matr{Q^{\text{expl}}}
  \end{pmatrix}
\end{multline*}
with
\begin{align*}
  \matr{M} &= (M_{ij})  &M_{ij} &= (\psi_i, \psi_j)_{\Gamma_h},\\
  \matr{A} &= (A_{ij}) & A_{ij}&= (\nabla \psi_i, \nabla
  \psi_j)_{\Gamma_h},\\
  \matr{A}^2 &= (A_{ij}^2) & A_{ij}^2 &= (A(\D\phi_h^{(m)})\nabla
  \psi_i, \nabla \psi_j)_{\Gamma_h}, \\ 
  \matr{F_u^{\text{impl}}}   &= ((F_u^{\text{impl}})_{ij}) & (F_u^{\text{impl}})_{ij} &=
  (\partial_uf(u_h^{(m)},v_h^{(m)}) \psi_i, \psi_j)_{\Gamma_h},\\
  \matr{F_v^{\text{impl}}}   &= ((F_v^{\text{impl}})_{ij}) & (F_v^{\text{impl}})_{ij} &=
  (\partial_vf(u_h^{(m)},v_h^{(m)}) \psi_i, \psi_j)_{\Gamma_h},\\
  \matr{Q_u^{\text{impl}}}   &= ((Q_u^{\text{impl}})_{ij}) & (Q_u^{\text{impl}})_{ij} &=
  (\partial_uq(u_h^{(m)},v_h^{(m)},V_h^{(m)}) \psi_i, \psi_j)_{\Gamma_h},\\
  \matr{Q_v^{\text{impl}}}   &= ((Q_v^{\text{impl}})_{ij}) & (Q_v^{\text{impl}})_{ij} &=
  (\partial_vq(u_h^{(m)},v_h^{(m)},V_h^{(m)}) \psi_i, \psi_j)_{\Gamma_h},\\
  \matr{F^{\text{expl}}}   &= (F^{\text{expl}}_{i}) & F^{\text{expl}}_{i} &=
  (F_e(u_h^{(m)},v_h^{(m)}), \psi_i)_{\Gamma_h},\\
  \matr{Q^{\text{expl}}}   &= (Q^{\text{expl}}_{i}) & Q^{\text{expl}}_{i} &=
  (Q_e(u_h^{(m)},v_h^{(m)}), \psi_i)_{\Gamma_h},
\end{align*}
where $(\cdot, \cdot)_{\Gamma_h}$ denotes $L^2$-scalar product. The above
linear system has to be solved in every time step, which is done by a
stabilized bi-conjugate gradient method (BiCGStab).

\subsection{Numerical Results}
First, we present numerical results reproducing the results of the
stability analysis in Section \ref{sec:turing}. To be more precise, we
choose a set of parameters fulfilling the conditions of Theorem
\ref{thm:turing} sufficient for instability, which is the basis of our
further numerical investigations:
\begin{equation}
\label{eq:parameters}
  d=1000;\;a_1 = 0; \; a_2 = 20; \; a_3 = 160; \; a_4 = 1; \; a_5 = 0.5; \; a_6
  = 0.1;\; a_{-6} = 1;\; \gamma = 400.
\end{equation}
Furthermore, we consider the unit-sphere $\Gamma = S^2$ and random initial
conditions $u_0, v_0: \Gamma_h \to [0,0.02]$ and $V_0 =
10$. Fig. \ref{fig:basic} shows the corresponding discrete solutions
$u_h, v_h$ at different times. A stationary pattern with a single spot appears.
\begin{figure}[here]
\hfill
\includegraphics*[width=0.14\textwidth]{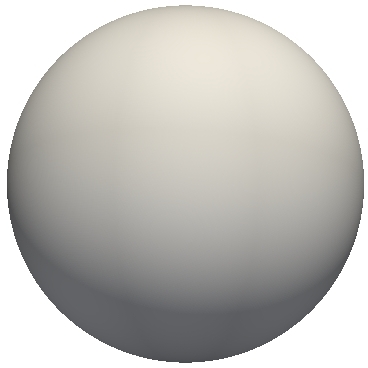}
\hfill
\includegraphics*[width=0.14\textwidth]{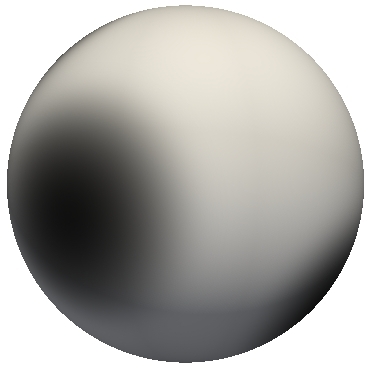}
\hfill
\includegraphics*[width=0.14\textwidth]{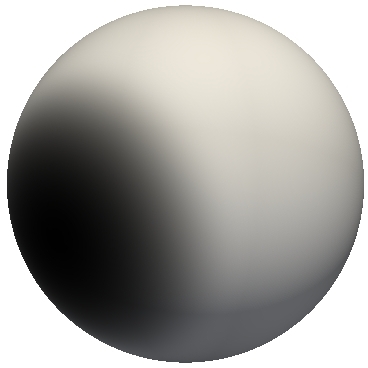}
\hfill
\includegraphics*[width=0.14\textwidth]{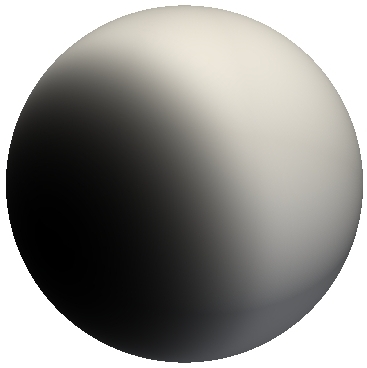}
\hfill
\includegraphics*[width=0.035\textwidth]{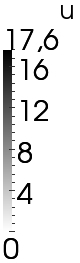}\\
\vspace{3ex}
\hfill
\includegraphics*[width=0.14\textwidth]{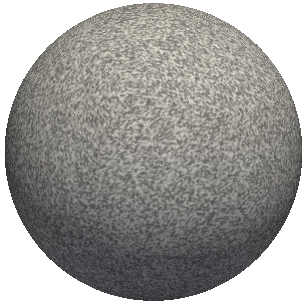}
\hfill
\includegraphics*[width=0.14\textwidth]{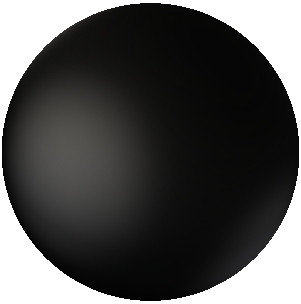}
\hfill
\includegraphics*[width=0.14\textwidth]{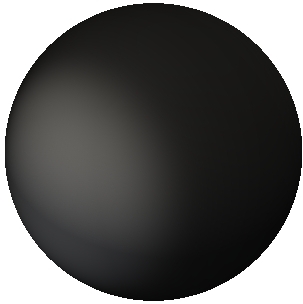}
\hfill
\includegraphics*[width=0.14\textwidth]{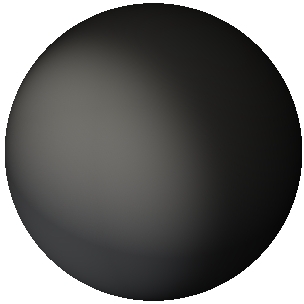}
\hfill
\includegraphics*[width=0.035\textwidth]{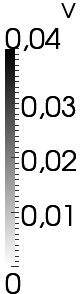}
\caption{\label{fig:basic}\em From left to right: the discrete
  solutions $u_h,v_h$ for $t=t_0=0$, $t = 0.5$, $t = 5$, and $t
  = 25$.}
\end{figure}

\subsubsection{Varying Parameters}

Based on the choice of parameters \eqref{eq:parameters} we investigate
the influence of varying parameters. First, we observe that
doubling the parameter $a_2$ leads to spatially homogeneous stationary
solutions, whereas halving $a_2$ leads to
stationary patterns with two maxima for $u_h$ (see Fig. \ref{fig:a2}).

\begin{figure}[here]
\includegraphics*[width=0.22\textwidth]{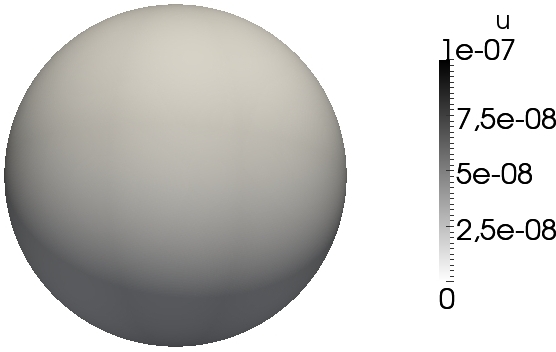}
\includegraphics*[width=0.22\textwidth]{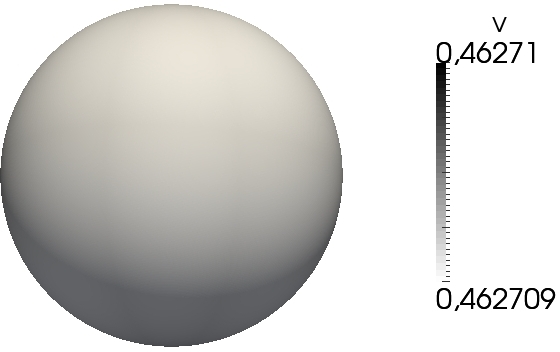}
\includegraphics*[width=0.22\textwidth]{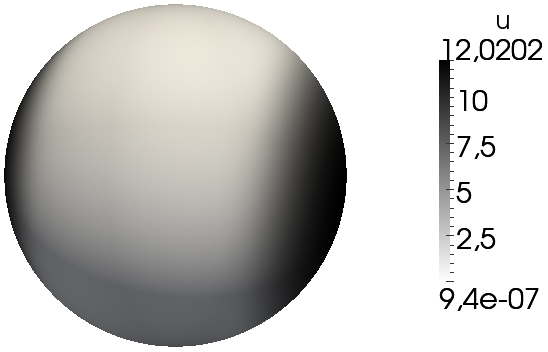}
\includegraphics*[width=0.22\textwidth]{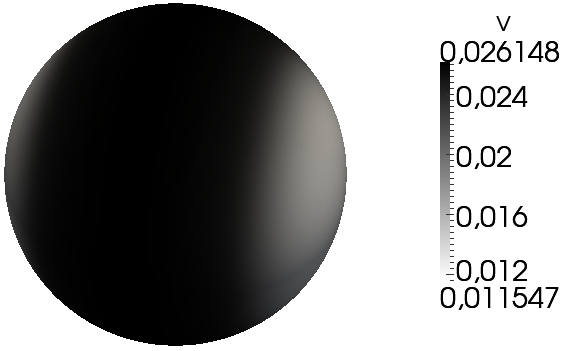}
\caption{\label{fig:a2}\em From left to right: the discrete
  solutions $u_h,v_h$ for $a_2=40$ (left) and $a_2=10$ (right) $t = 25$.}
\end{figure}

Additionally, we observe that halving the parameter $a_3$ leads to
spatially homogeneous stationary solutions, whereas doubling $a_3$ leads to
stationary patterns with two maxima for $u_h$ (see Fig. \ref{fig:a3}).

\begin{figure}[here]
\includegraphics*[width=0.22\textwidth]{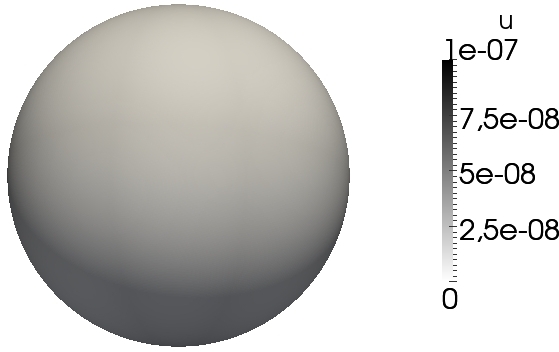}
\includegraphics*[width=0.22\textwidth]{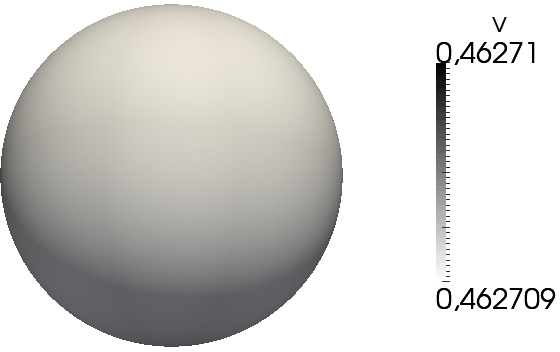}
\includegraphics*[width=0.22\textwidth]{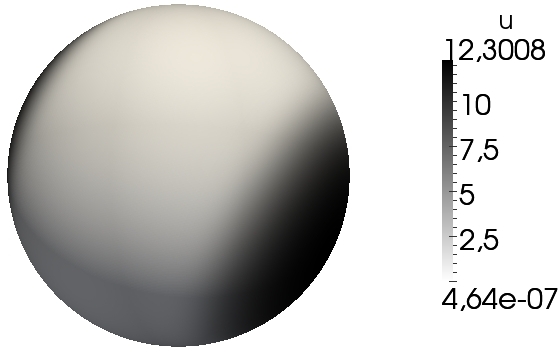}
\includegraphics*[width=0.22\textwidth]{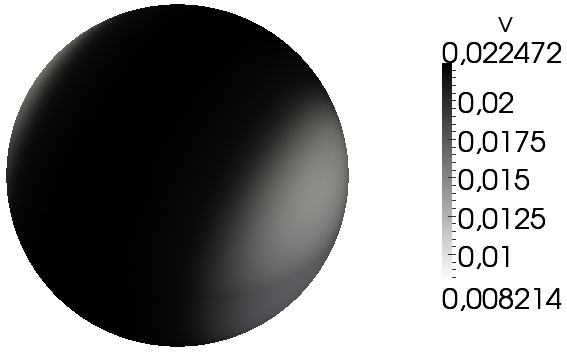}
\caption{\label{fig:a3}\em From left to right: the discrete
  solutions $u_h,v_h$ for $a_3=80$ (left) and $a_3=320$ (right) $t = 25$.}
\end{figure}

Another interesting behavior can be seen by varying the diffusion
constant $d$. While the estimates \eqref{cdt:d1} and \eqref{cdt:d2}
lead to a condition $d \gtrsim 790$ sufficient for instability, an exact
computation yields a maximal diffusion constant $d_c \approx 101$
satisfying equality in one of relations \eqref{eq:rd-Tu3},
\eqref{eq:rd-Tu4}. This is reproduced in Fig. \ref{fig:d} showing
stability for $d=100$ and instability for $d=105$.

\begin{figure}[here]
\includegraphics*[width=0.22\textwidth]{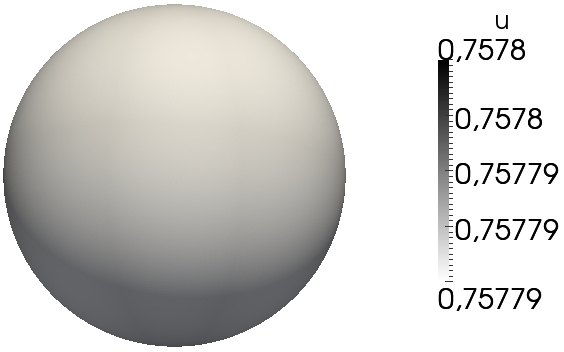}
\includegraphics*[width=0.22\textwidth]{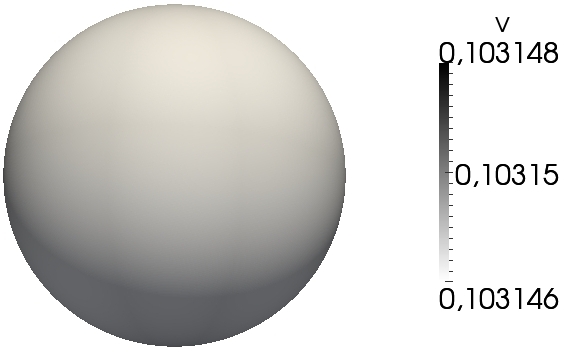}
\includegraphics*[width=0.22\textwidth]{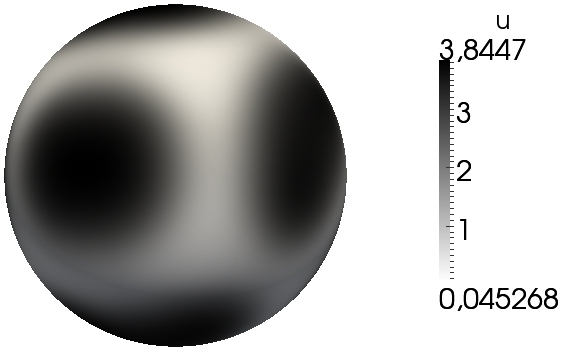}
\includegraphics*[width=0.22\textwidth]{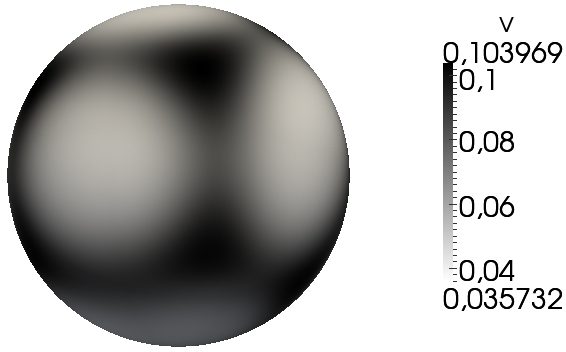}
\caption{\label{fig:d}\em From left to right: the discrete
  solutions $u_h,v_h$ for $d=100$ (left) and $d=105$ (right) $t = 25$.}
\end{figure}

\subsection{Comparison with `realistic' parameter ranges}
A full set of realistic parameters is not available, but there are
several \emph{in vivo} and \emph{in vitro} measurements giving some
estimates or average values. For the case of the Cdc42 GTPase cycle in
yeast cells we have collected data from \cite{GoPo08}, \cite{Jilk03},
\cite{KaCh07}, and \cite{GaGJ07} and evaluated our model for the
following set of values,
\begin{align*}
 &k_1 := 1.056831769\cdot 10^{-8} ,\quad k_2 := 0.1056831769\cdot
 10^{-5} ,\quad k_3 := 946.2243938\\ & k_4 := 18.92448788 ,\quad k_{5}
 := 0.1056831769\cdot 10^{-2} ,\quad k_{-5} := 0.3 ,\quad b_6 :=
 0.3170495307\cdot 10^{-1} \\& b_{-6} := 0.133 ,\quad g_0 :=
 37848.97575 ,\quad d_u := 2.5 \cdot 10^{-15} ,\quad c_{max} :=
 47311.21969 ,\\ &V_0 := 4.894264108\cdot10^{10}, 
\end{align*}
where the units are as given in Section \ref{sec:model}. 

We find for these values that a homogeneous stationary state of
activator--substrate-depletion type in fact exists. For a Turing-type
instability we need a lateral diffusion for activated GTPase of order
$10^{-8}\text{m}^2\text{s}^{-1}$, which is unrealistically
large. Nevertheless assuming such a value we obtain a condition on the
spatial scale: We find that $R$ has to be at least of order
$10^{-6}\text{m}$, which is close to the typical diameter of a yeast
cell. We therefore see that the critical condition is in fact the
large difference in diffusion.

The uncertainty in the parameter values is quite large: there are not
enough data for individual proteins available and the parameters
chosen above are therefore a composite of available data. Many
measurements are taken \emph{in vitro} and are only an estimate for
\emph{in vivo} conditions. Comparing different sources and different
GTPases one may find variations up to order 10 in the 
parameters. Within the set of parameter values allowing for Turing instabilities we find values that are within the range of realistic values, except that the ratio $d$ of lateral diffusion values has to be of order at least $10^2$. Such a large difference in free lateral diffusion for
active and inactive GTPase seems unrealistic. However, heterogeneities
of the cell membrane may lead to differences in the `effective'
diffusion speed, see the discussion below. 
\section{Discussion}\label{sec:discussion}
The GTPase cycle presents an example for a coupled system of processes in the inner volume and on the outer membrane. We have proposed a mathematical model in the form of a fully coupled PDE system. A two-variable reduction yields a \emph{non-local} reaction--diffusion system on the membrane. With the interest in finding mechanisms that support the emergence of cell polarity we have investigated pattern forming properties. We have shown that the reduced model in principle supports Turing type diffusive instabilities, but -- as for general local RD systems -- needs large differences in diffusion constants. In numerical simulations we have confirmed our theoretical findings and have explored the type of pattern produced and the influence of parameter changes. Both the formulation of the model and the numerical schemes are prepared to investigate more complex models and to incorporate additional features.

In similar but different models diffusive instabilities have been shown to exist \cite{BaJo05}, \cite{GoPo08}, or not to exist \cite{AAWW08}, \cite{BDKR07} (unless a phase separation force is added). Particularly interesting is a comparison with the work of Goryachev and Pokhilko \cite{GoPo08}. Their model is more detailed in the set of variables they consider and also accounts -- at least partially -- for the different dimensionalities of the processes: the membrane is thought as a thin compartment with positive volume and the cell as an adjacent bigger compartment. Concentrations on the membrane and in the inner cell are weighed with a factor that accounts for the different sizes of the volumes.  A major difference to our work is that the mathematical analysis in \cite{GoPo08} treats all variables on one common domain of definition (for both cytosolic and membrane-bound quantities). In our approach on the other hand we distinguish explicitly between the cytosolic GTPase variable $V$ defined in $B$ and the membrane variables defined on $\Gamma$, which then makes laws for fluxes from the cytosol to the membrane necessary and meaningful.  Nevertheless there are more similarities between these two approaches. Also in \cite{GoPo08} a non-local reduction to a two-variable system is given. The substrate there is represented by the sum of cytosolic and membrane bound GTPase and inherits a higher diffusion constant than the solely membrane bound active form. In view of pattern forming properties their model then produces Turing instabilities, in more realistic parameter ranges than our model.

One of the main features of our reduced model is that large differences in the lateral diffusion constant for active and inactive GTPase are required to obtain a Turing instability. Mathematical models with a simpler (but more artificial) dimensional coupling on the other hand do support diffusive instabilities. However, in these studies differences in lateral diffusion were put more directly into the model and Turing instabilities then appear as a consequence of this model design. Our analysis demonstrates that it is not clear whether large cytosolic diffusion is sufficient for cell polarization by a Turing mechanism. To clarify this more detailed studies are necessary. The absence of realistic Turing patterns in our model might be a consequence of our reduction to the infinite cytosolic diffusion limit. This leads to constant cytosolic GTPase concentration, whereas the Turing mechanism intimately relies on spatial heterogeneity. We have performed numerical experiments to compare the behavior of the fully coupled system and the reduced model. For large but finite cytosolic diffusion our simulations for the full model agreed qualitatively with corresponding simulations for the reduction. However, it is difficult to numerically explore the Turing space of the fully coupled model and error estimates for differences of solutions to the reduced model and the three-variable system are at present not available. 

In our model various properties of the GTPase cycle in living cells have been neglected that in principle may contribute to polarization. 
In a recent paper by Butler and Goldenfeld \cite{BuGo11} it was shown that the inclusion of intrinsic noise in reaction-diffusion systems can lead to the formation of Turing-type  `quasipatterns' for parameter values substantially different from the Turing space for the corresponding noise-free system. The heterogeneity of  the plasma membrane might influence polarization in living cells: Microdomains with different lipid composition are present, and active and inactive forms of GTPase possibly associate with different preferences to distinct microdomains. For the case of Ras GTPase in \cite{LSSS05} a severe reduction in Ras mobility has been observed upon activation. This effect might introduce a difference in diffusion sufficient for Turing patterns. Finally, other mechanisms different from Turing pattern formation might be responsible for cell polarization. Possible candidates are a wave pinning mechanism \cite{MoJE08}, or an intrinsically stochastic mechanism that relies on a particle based approach (as proposed in \cite{AAWW08}) instead of a continuous model. 

A better understanding of symmetry breaking properties in models where systems of different dimensionalities are coupled remains important. We have proposed a more detailed coupling model that presents a good basis for future extensions. A stability analysis of the fully coupled model and the inclusion of possible additional contributions to cell polarization are interesting tasks for further studies that might lead to a more complete picture of the origin of cell polarization.
\begin{appendix}
\section{Turing instability}
\label{app:turing}
Since in our GTPase cycle model the classical conditions for Turing type pattern have to modified by the non-locality of the source term, we briefly outline the analysis of the classical Turing mechanism. We follow here \cite[Section 5.3]{Jost07}. Let a spatial domain $\Omega\subset \Rn$ be given and consider a system of reaction--diffusion equations
\begin{align}
	\partial_t u\,&=\, \Delta u + \gamma f(u,v), \label{eq:rd-1}\\
	\partial_t v\,&=\, d\Delta v + \gamma g(u,v), \label{eq:rd-2}
\end{align}
where $u=u(x,t)$, $v=v(x,t)$, $d>1$, $\gamma>0$ and where $f:\R^2\to\R$, $g:\R^2\to\R$ are given functions. We complement \eqref{eq:rd-1}, \eqref{eq:rd-2} first by initial conditions
\begin{gather*}
	u(x,0)\,=\,u^0(x),\qquad v(x,0)\,=\, v^0(x),
\end{gather*}
for given $u^0,v^0:\Omega\,\to\,\R$, and second by zero Neumann-boundary data
\begin{gather*}
	\nabla u\cdot\nu_\Omega u\,=\, \nabla v\cdot\nu_\Omega\,=\, 0\qquad\text{ on }\partial\Omega,
\end{gather*}
where $\nu_\Omega$ denotes the outer normal of $\Omega$. With this boundary conditions the following analysis carries immediately over to the case that the spatial domain is given by a compact closed smooth hypersurface in $\Rn$, with the only difference that the Laplace operator has to be replaced by the surface Laplace--Beltrami operator.

The Turing mechanism is described by a stationary point $(u_*,v_*)$ that is spatially homogeneous and linearly stable under spatially homogeneous perturbation, but that is linearly unstable under heterogeneous perturbations. We therefore consider now $(u_*,v_*) \in\R^2$ with 
\begin{gather*}
	f(u_*,v_*)\,=\,0,\qquad g(u_*,v_*)\,=\, 0.
\end{gather*}
For spatially homogeneous solutions \eqref{eq:rd-1}, \eqref{eq:rd-2} reduce to the ODE system
\begin{align}
	\partial_t u\,&=\, \gamma f(u,v), \label{eq:rd-ode-1}\\
	\partial_t v\,&=\, \gamma g(u,v). \label{eq:rd-ode-2}
\end{align}
The condition of linear stability then reduces to the conditions that the trace of the Jacobian $D(f,g)$ is negative and that the determinant of $D(f,g)$ is positive, \emph{i.e.}
\begin{align}
	\partial_uf(u_*,v_*) + \partial_vg(u_*,v_*)\,&<\, 0, \label{eq:rd-Tu1}\\
	\partial_uf(u_*,v_*)\partial_vg(u_*,v_*) - \partial_vf(u_*,v_*)\partial_ug(u_*,v_*)\,&>\,0. \label{eq:rd-Tu2}
\end{align}
The linearization of \eqref{eq:rd-1}, \eqref{eq:rd-2} in $(u_*,v_*)$ for perturbations in direction of arbitrary smooth functions $\varphi,\psi:\Omega\times(0,T)\to\R$ is given by the system
\begin{gather}
	\partial_t\begin{pmatrix} \varphi\\ \psi \end{pmatrix} \,=\, \begin{pmatrix} 1 & 0 \\ 0 &d\end{pmatrix}\begin{pmatrix} \varphi\\ \psi \end{pmatrix} + \gamma \begin{pmatrix} \partial_u f(u_*,v_*) & \partial_v f(u_*,v_*) \\ \partial_u g(u_*,v_*) & \partial_v g(u_*,v_*)\end{pmatrix}\begin{pmatrix} \varphi\\ \psi \end{pmatrix}. \label{eq:rd-lin}
\end{gather}
We next take the complete orthonormal basis $(w_j)_{j\in\N_0}$ of $L^2(\Gamma)$ given by eigenvectors of the Laplacian with respect to zero Neumann boundary data,
\begin{align}
	-\Delta w_j \,&=\, \lambda_j w_j\quad \text{ in }\Omega, \label{eq:eigen}\\
	\nabla w_j\cdot \nu_\Omega \,&=\, 0 \quad\text{ on }\partial\Omega, \label{eq:eigen-bdry}
\end{align}
for $j=0,1,2,\ldots$ and where $\lambda_0=0<\lambda_1\leq\lambda_2\leq\ldots$ denote the corresponding eigenvalues. By the orthonormality condition and \eqref{eq:eigen}, \eqref{eq:eigen-bdry} we have
\begin{align}
	\|w_j\|_{L^2(\Omega)}\,&=\,1\quad &\text{ for all }j\in\N_0, \label{eq:app-modes}\\
	\int_\Omega w_j w_i \,dx \,&=\, \int_\Omega \nabla w_j\cdot\nabla w_i \,dx \,=\, 0 \quad &\text{ for all }i,j\in\N_0, i\neq j. \label{eq:app-onb}
\end{align}
We then decompose $\varphi(\cdot,t),\psi(\cdot,t)$ with respect to this orthonormal basis,
\begin{align*}
	\varphi(x,t)\,&=\, \alpha_0(t)w_0 + \sum_{j\in\N} \beta_j(t)w_j(x),\\
	\psi(x,t)\,&=\, \beta_0(t)w_0 + \sum_{j\in\N} \beta_j(t)w_j(x).
\end{align*}
Inserting this representation in \eqref{eq:u-lin}, \eqref{eq:v-lin} and taking the $L^2(\Gamma)$ scalar product with $w_i$ by the orthonormality and \eqref{eq:app-modes} the equation \eqref{eq:rd-lin} decomposes into linear systems
\begin{gather}
	\partial_t\begin{pmatrix} \alpha_i\\ \beta_i \end{pmatrix} \,=\, -\lambda_i \begin{pmatrix} 1 & 0 \\ 0 &d\end{pmatrix} \begin{pmatrix} \alpha_i\\ \beta_i \end{pmatrix} + \gamma \begin{pmatrix} \partial_u f(u_*,v_*) & \partial_v f(u_*,v_*) \\ \partial_u g(u_*,v_*) & \partial_v g(u_*,v_*)\end{pmatrix}\begin{pmatrix} \alpha_i\\ \beta_i \end{pmatrix} \label{eq:app-alpha-beta}
\end{gather}
for $i=0,1,\ldots$. The case $i=0$ corresponds to the case of spatially homogeneous perturbations. In order to have an instability of \eqref{eq:rd-1}, \eqref{eq:rd-2} we therefore need that for an $i\in\N$ \eqref{eq:app-alpha-beta} is unstable. This gives the necessary conditions \cite[Theorem 5.3.1]{Jost07}
\begin{align}
	d\partial_uf(u_*,v_*) + \partial_vg(u_*,v_*)\,&>\, 0, \label{eq:rd-Tu3}\\
	\big(d\partial_uf(u_*,v_*) + \partial_vg(u_*,v_*)\big)^2 -4d \Big(\partial_uf(u_*,v_*)\partial_vg(u_*,v_*) - \partial_vf(u_*,v_*)\partial_ug(u_*,v_*)\Big)\,&>\,0. \label{eq:rd-Tu4}
\end{align}
In order to have an instability under these conditions it is sufficient that there exists an eigenvalue $\lambda_i$, $i\in\N$, such that
\begin{gather*}
	\mu_-\,<\,\lambda_i\,<\, \mu_+
\end{gather*}
where $\mu=\mu_\pm$ are the roots of the quadratic equation
\begin{gather*}
	d\mu^2 -\gamma \big(d\partial_uf(u_*,v_*) + \partial_vg(u_*,v_*)\big)\mu + \gamma^2 \big(\partial_uf(u_*,v_*)\partial_vg(u_*,v_*) - \partial_vf(u_*,v_*)\partial_ug(u_*,v_*)\big) \,=\, 0.
\end{gather*}

\end{appendix}


\begin{thebibliography}{10}


\providecommand{\url}[1]{\texttt{#1}}
\expandafter\ifx\csname urlstyle\endcsname\relax
  \providecommand{\doi}[1]{doi: #1}\else
  \providecommand{\doi}{doi: \begingroup \urlstyle{rm}\Url}\fi

\bibitem[1]{AAWW08}
\textsc{Altschuler}, Steven~J. ; \textsc{Angenent}, Sigurd~B. ; \textsc{Wang},
  Yanqin  ; \textsc{Wu}, Lani~F.:
\newblock On the spontaneous emergence of cell polarity.
\newblock {In: }\emph{Nature} 454 (2008), Aug, Nr. 7206, 886--889.
\newblock \url{http://dx.doi.org/10.1038/nature07119}. --
\newblock DOI 10.1038/nature07119

\bibitem[2]{BoRW07}
\textsc{Bos}, Johannes~L. ; \textsc{Rehmann}, Holger  ; \textsc{Wittinghofer},
  Alfred:
\newblock GEFs and GAPs: critical elements in the control of small G proteins.
\newblock {In: }\emph{Cell} 129 (2007), Jun, Nr. 5, 865--877.
\newblock \url{http://dx.doi.org/10.1016/j.cell.2007.05.018}. --
\newblock DOI 10.1016/j.cell.2007.05.018

\bibitem[3]{BDKR07}
\textsc{Brusch}, Lutz ; \textsc{Del Conte-Zerial}, Perla ;
  \textsc{Kalaidzidis}, Yannis ; \textsc{Rink}, Jochen ; \textsc{Habermann},
  Bianca ; \textsc{Zerial}, Marino  ; \textsc{Deutsch}, Andreas:
\newblock Protein Domains of {GTP}ases on Membranes: Do They Rely on Turing’s
  Mechanism?
\newblock {In: }\textsc{Deutsch}, Andreas (Hrsg.) ; \textsc{Brusch}, Lutz
  (Hrsg.) ; \textsc{Byrne}, Helen (Hrsg.) ; \textsc{Vries}, Gerda de (Hrsg.)  ;
  \textsc{Herzel}, Hanspeter (Hrsg.): \emph{Mathematical Modeling of Biological
  Systems, Volume I}.
\newblock Birkh\"auser Boston, 2007

\bibitem[4]{BuGo11}
\textsc{Butler}, Thomas ; \textsc{Goldenfeld}, Nigel:
\newblock Fluctuation-driven Turing patterns.
\newblock {In: }\emph{Phys. Rev. E} 84 (2011), Jul, 011112.
\newblock \url{http://dx.doi.org/10.1103/PhysRevE.84.011112}. --
\newblock DOI 10.1103/PhysRevE.84.011112

\bibitem[5]{DiHS10}
\textsc{Dierkes}, U. ; \textsc{Hildebrandt}, S.  ; \textsc{Sauvigny}, F.:
\newblock \emph{Minimal Surfaces}.
\newblock Springer, 2010 (Grundlehren der mathematischen Wissenschaften Series
  pt. 1). --
\newblock ISBN 9783642116971

\bibitem[6]{DzEl07}
\textsc{Dziuk}, G. ; \textsc{Elliott}, C.~M.:
\newblock Finite elements on evolving surfaces.
\newblock {In: }\emph{IMA J. Numer. Anal.} 27 (2007), S. 262--292

\bibitem[7]{GaGJ07}
\textsc{Garmendia-Torres}, Cecilia ; \textsc{Goldbeter}, Albert  ;
  \textsc{Jacquet}, Michel:
\newblock Nucleocytoplasmic oscillations of the yeast transcription factor
  Msn2: evidence for periodic PKA activation.
\newblock {In: }\emph{Curr Biol} 17 (2007), Jun, Nr. 12, 1044--1049.
\newblock \url{http://dx.doi.org/10.1016/j.cub.2007.05.032}. --
\newblock DOI 10.1016/j.cub.2007.05.032

\bibitem[8]{GoRa05}
\textsc{Goody}, R.~S. ; \textsc{Rak}, A.  ; \textsc{Alexandrov}, K.:
\newblock The structural and mechanistic basis for recycling of Rab proteins
  between membrane compartments.
\newblock {In: }\emph{Cell Mol Life Sci} 62 (2005), Aug, Nr. 15, 1657--1670.
\newblock \url{http://dx.doi.org/10.1007/s00018-005-4486-8}. --
\newblock DOI 10.1007/s00018--005--4486--8

\bibitem[9]{GoPo08}
\textsc{Goryachev}, Andrew~B. ; \textsc{Pokhilko}, Alexandra~V.:
\newblock Dynamics of Cdc42 network embodies a Turing-type mechanism of yeast
  cell polarity.
\newblock {In: }\emph{FEBS Lett} 582 (2008), Apr, Nr. 10, 1437--1443.
\newblock \url{http://dx.doi.org/10.1016/j.febslet.2008.03.029}. --
\newblock DOI 10.1016/j.febslet.2008.03.029

\bibitem[10]{GrON06}
\textsc{Grosshans}, Bianka~L. ; \textsc{Ortiz}, Darinel  ; \textsc{Novick},
  Peter:
\newblock Rabs and their effectors: achieving specificity in membrane traffic.
\newblock {In: }\emph{Proc Natl Acad Sci U S A} 103 (2006), Aug, Nr. 32,
  11821--11827.
\newblock \url{http://dx.doi.org/10.1073/pnas.0601617103}. --
\newblock DOI 10.1073/pnas.0601617103

\bibitem[11]{GuAG05}
\textsc{Guo}, Zhong ; \textsc{Ahmadian}, Mohammad~R.  ; \textsc{Goody},
  Roger~S.:
\newblock Guanine nucleotide exchange factors operate by a simple allosteric
  competitive mechanism.
\newblock {In: }\emph{Biochemistry} 44 (2005), Nov, Nr. 47, 15423--15429.
\newblock \url{http://dx.doi.org/10.1021/bi0518601}. --
\newblock DOI 10.1021/bi0518601

\bibitem[12]{JaHa05}
\textsc{Jaffe}, Aron~B. ; \textsc{Hall}, Alan:
\newblock RHO GTPASES: Biochemistry and Biology.
\newblock {In: }\emph{Annual Review of Cell and Developmental Biology} 21
  (2005), Nr. 1, 247-269.
\newblock \url{http://dx.doi.org/10.1146/annurev.cellbio.21.020604.150721}. --
\newblock DOI 10.1146/annurev.cellbio.21.020604.150721

\bibitem[13]{Jilk03}
\textsc{Jilkine}, Alexandra:
\newblock \emph{Mathematical Study of Rho GTPases in Motile Cells}, The
  University of British Columbia, Diss., 2003

\bibitem[14]{BaJo05}
\textsc{John}, Karin ; \textsc{B\"ar}, Markus:
\newblock Alternative mechanisms of structuring biomembranes: self-assembly
  versus self-organization.
\newblock {In: }\emph{Phys Rev Lett} 95 (2005), Nov, Nr. 19, S. 198101

\bibitem[15]{Jost07}
\textsc{Jost}, J{\"u}rgen:
\newblock \emph{Graduate Texts in Mathematics}. Bd. 214: {\emph{Partial
  differential equations}}.
\newblock Second.
\newblock New York : Springer, 2007. --
\newblock  xiv+356 S.
\newblock \url{http://dx.doi.org/10.1007/978-0-387-49319-0}.
\newblock \url{http://dx.doi.org/10.1007/978-0-387-49319-0}. --
\newblock ISBN 978--0--387--49318--3; 0--387--49318--2

\bibitem[16]{KaCh07}
\textsc{Katanaev}, Vladimir~L. ; \textsc{Chornomorets}, Matey:
\newblock Kinetic diversity in G-protein-coupled receptor signalling.
\newblock {In: }\emph{Biochem J} 401 (2007), Jan, Nr. 2, 485--495.
\newblock \url{http://dx.doi.org/10.1042/BJ20060517}. --
\newblock DOI 10.1042/BJ20060517

\bibitem[17]{Kell09}
\textsc{Keller}, J{\"u}rgen~U.:
\newblock An Outlook on Biothermodynamics. II. Adsorption of Proteins.
\newblock {In: }\emph{Journal of Non-Equilibrium Thermodynamics} 34 (2009),
  M{\^^b a}rz, Nr. 1, 1--33.
\newblock \url{http://dx.doi.org/10.1515/JNETDY.2009.001}. --
\newblock ISSN 0340--0204

\bibitem[18]{KoMe94}
\textsc{Koch}, A.~J. ; \textsc{Meinhardt}, H.:
\newblock Biological pattern formation: from basic mechanisms to complex
  structures.
\newblock {In: }\emph{Rev. Mod. Phys.} 66 (1994), Oct, Nr. 4, S. 1481--1507.
\newblock \url{http://dx.doi.org/10.1103/RevModPhys.66.1481}. --
\newblock DOI 10.1103/RevModPhys.66.1481

\bibitem[19]{LaVo10}
\textsc{Landsberg}, C. ; \textsc{Voigt}, A.:
\newblock A multigrid finite element method for reaction-diffusion systems on
  surfaces.
\newblock {In: }\emph{Comp. Vis. Sci.} 13 (2010), Nr. 4, S. 177--185

\bibitem[20]{LMRR01}
\textsc{Lippé}, R. ; \textsc{Miaczynska}, M. ; \textsc{Rybin}, V. ;
  \textsc{Runge}, A.  ; \textsc{Zerial}, M.:
\newblock Functional synergy between Rab5 effector Rabaptin-5 and exchange
  factor Rabex-5 when physically associated in a complex.
\newblock {In: }\emph{Mol Biol Cell} 12 (2001), Jul, Nr. 7, S. 2219--2228

\bibitem[21]{LSSS05}
\textsc{Lommerse}, Piet H~M. ; \textsc{Snaar-Jagalska}, B~E. ; \textsc{Spaink},
  Herman~P.  ; \textsc{Schmidt}, Thomas:
\newblock Single-molecule diffusion measurements of H-Ras at the plasma
  membrane of live cells reveal microdomain localization upon activation.
\newblock {In: }\emph{J Cell Sci} 118 (2005), May, Nr. Pt 9, 1799--1809.
\newblock \url{http://dx.doi.org/10.1242/jcs.02300}. --
\newblock DOI 10.1242/jcs.02300

\bibitem[22]{MoJE08}
\textsc{Mori}, Yoichiro ; \textsc{Jilkine}, Alexandra  ;
  \textsc{Edelstein-Keshet}, Leah:
\newblock Wave-pinning and cell polarity from a bistable reaction-diffusion
  system.
\newblock {In: }\emph{Biophys J} 94 (2008), May, Nr. 9, 3684--3697.
\newblock \url{http://dx.doi.org/10.1529/biophysj.107.120824}. --
\newblock DOI 10.1529/biophysj.107.120824

\bibitem[23]{Murr90}
\textsc{Murray}, J.D.:
\newblock Discussion: Turing's theory of morphogenesis--Its influence on
  modelling biological pattern and form.
\newblock {In: }\emph{Bulletin of Mathematical Biology} 52 (1990), Nr. 1-2, 119
  - 152.
\newblock \url{http://dx.doi.org/DOI: 10.1016/S0092-8240(05)80007-2}. --
\newblock DOI DOI: 10.1016/S0092--8240(05)80007--2. --
\newblock ISSN 0092--8240

\bibitem[24]{NiPr77}
\textsc{Nicolis}, G. ; \textsc{Prigogine}, I.:
\newblock \emph{{Self-organization in nonequilibrium systems: from dissipative
  structures to order through fluctuations}}.
\newblock Wiley, 1977 (A Wiley-Interscience publication). --
\newblock ISBN 9780471024019

\bibitem[25]{PaBi07}
\textsc{Park}, Hay-Oak ; \textsc{Bi}, Erfei:
\newblock \emph{Central roles of small GTPases in the development of cell
  polarity in yeast and beyond}.
\newblock Washington, DC, ETATS-UNIS, 2007. --
\newblock Anglais

\bibitem[26]{Pfef03}
\textsc{Pfeffer}, Suzanne:
\newblock Membrane domains in the secretory and endocytic pathways.
\newblock {In: }\emph{Cell} 112 (2003), Feb, Nr. 4, S. 507--517

\bibitem[27]{PfAi04}
\textsc{Pfeffer}, Suzanne ; \textsc{Aivazian}, Dikran:
\newblock Targeting Rab GTPases to distinct membrane compartments.
\newblock {In: }\emph{Nat Rev Mol Cell Biol} 5 (2004), Nov, Nr. 11, 886--896.
\newblock \url{http://dx.doi.org/10.1038/nrm1500}. --
\newblock DOI 10.1038/nrm1500

\bibitem[28]{PBLH04}
\textsc{Postma}, Marten ; \textsc{Bosgraaf}, Leonard ; \textsc{Loovers},
  Harriët~M.  ; \textsc{Haastert}, Peter J M~V.:
\newblock Chemotaxis: signalling modules join hands at front and tail.
\newblock {In: }\emph{EMBO Rep} 5 (2004), Jan, Nr. 1, 35--40.
\newblock \url{http://dx.doi.org/10.1038/sj.embor.7400051}. --
\newblock DOI 10.1038/sj.embor.7400051

\bibitem[29]{SRNR00}
\textsc{S{\"o}nnichsen}, B. ; \textsc{Renzis}, S.~D. ; \textsc{Nielsen}, E. ;
  \textsc{Rietdorf}, J.  ; \textsc{Zerial}, M.:
\newblock Distinct membrane domains on endosomes in the recycling pathway
  visualized by multicolor imaging of Rab4, Rab5, and Rab11.
\newblock {In: }\emph{J Cell Biol} 149 (2000), May, Nr. 4, S. 901--914

\bibitem[30]{TaSM01}
\textsc{Takai}, Y ; \textsc{Sasaki}, T  ; \textsc{Matozaki}, T:
\newblock Small GTP-binding proteins.
\newblock {In: }\emph{Physiological Reviews} 81 (2001), Nr. 1, 153-208.
\newblock \url{http://www.ncbi.nlm.nih.gov/pubmed/11152757}

\bibitem[31]{Turi52}
\textsc{Turing}, Alan~M.:
\newblock The Chemical Basis of Morphogenesis.
\newblock {In: }\emph{Philosophical Transactions of the Royal Society of
  London. Series B, Biological Sciences} 237 (1952), Nr. 641, 37--72.
\newblock \url{http://www.jstor.org/stable/92463}. --
\newblock ISSN 00804622

\bibitem[32]{VeVo07}
\textsc{Vey}, S. ; \textsc{Voigt}, A.:
\newblock A{MDiS} --- Adaptive multidimensional simulations.
\newblock {In: }\emph{Comput. Visual. Sci.} 10 (2007), S. 57--67

\bibitem[33]{WAWL03}
\textsc{Wedlich-Soldner}, Roland ; \textsc{Altschuler}, Steve ; \textsc{Wu},
  Lani  ; \textsc{Li}, Rong:
\newblock Spontaneous Cell Polarization Through Actomyosin-Based Delivery of
  the Cdc42 GTPase.
\newblock {In: }\emph{Science} 299 (2003), Nr. 5610, 1231-1235.
\newblock \url{http://dx.doi.org/10.1126/science.1080944}. --
\newblock DOI 10.1126/science.1080944

\end{thebibliography}

\end{document}